\documentclass{amsart}

\numberwithin{equation}{section}
\newtheorem{theorem}{Theorem}
\newtheorem{lemma}{Lemma}

\newtheorem{definition}{Definition}
\newtheorem{corollary}{Corollary}
\newtheorem{remark}{Remark}

\numberwithin{lemma}{section}
\numberwithin{proposition}{section}
\numberwithin{corollary}{section}
\numberwithin{theorem}{section}
\numberwithin{equation}{section}
\numberwithin{example}{section}

\begin{document}

\title
  {\bf The Cauchy problem of the Ward equation\\
  }

\author   { Derchyi Wu }
\maketitle

\begin{center}
INSTITUTE OF MATHEMATICS
\end{center}

\begin{center}
ACADEMIA SINICA
\end{center}

\begin{center}
TAIPEI, TAIWAN, R. O. C.
\end{center}

\begin{center}
mawudc@math.sinica.edu.tw
\end{center}

\section*{{Abstract}}
We generalize the results of \cite{V90}, \cite{FI01}, \cite{DTU} to study the inverse scattering problem of the Ward equation with non-small data and solve the Cauchy problem of the Ward equation with a  non-small purely continuous scattering data. 

\vskip.25in

\tableofcontents 

%%%%%%%%%%%%%%%%%%%%%%%%%%%%%%%%%%%%%%%%%%
\section{Introduction }
%%%%%%%%%%%%%%%%%%%%%%%%%%%%%%%%%%%%%%%%%%%

\vskip.2in
The Ward equation (or the modified $2+1$ chiral model) 
\begin{equation}
\partial_t\left(J^{-1}\partial_t J\right)-\partial_x\left(J^{-1}\partial_xJ\right)-\partial_y\left(J^{-1}\partial_yJ\right)-\left[J^{-1}\partial_tJ,J^{-1}\partial_yJ\right]=0,\label{E:Ward}
\end{equation}
for $J:\mathbb{R}^{2,1}\to SU(n)$, $\partial_w=\partial/\partial w$, is obtained from a dimension reduction and a gauge fixing of the self-dual Yang-Mills equation on $\mathbb{R}^{2,2}$ \cite{DTU}, \cite{T06}. It is an integrable system which possesses the Lax pair \cite{W88}
\begin{equation}
\left[\lambda\partial_x-\partial_\xi-J^{-1}\partial_\xi J,\lambda\partial_\eta-\partial_x-J^{-1}\partial_x J\right]=0\label{E:Lax3}
\end{equation}
with $\xi=\frac{t+y}2$, $\eta=\frac{t-y}2$.  Note (\ref{E:Lax3}) implies that $J^{-1}\partial_\xi J=-\partial_x Q$, $J^{-1}\partial_x J=-\partial_\eta Q$. Then by a change of variables $(\eta,x,\xi)\rightarrow (x,y,t)$, (\ref{E:Lax3}) is equivalent to
\begin{eqnarray}
&&(\partial_y-\lambda\partial_x)\Psi(x,y,t,\lambda)=\left(\partial_x Q(x,y,t)\right)\Psi(x,y,t,\lambda), \label{E:Lax01}\\
&&(\partial_t-\lambda^2\partial_x)\Psi(x,y,t,\lambda)=\left(\lambda \partial_xQ+\partial_yQ\right)\Psi(x,y,t,\lambda) \label{E:Lax02}
\end{eqnarray}
\cite{FI01}, and the Ward equation (\ref{E:Ward}) turns into:
\begin{equation}
\partial_x\partial_tQ=\partial^2_yQ+\left[\partial_yQ,\partial_xQ\right].\label{E:chiral}
\end{equation}

The construction of solitons, the study of the scattering properties of solitons, and  Darboux transformation of the Ward equation have been studied intensively by solving the degenerated Riemann-Hilbert problem and studying the limiting method 
\cite{W88}, \cite{W95}, \cite{I96}, \cite{An97}, \cite{An98}, \cite{IZ98}, \cite{Z93}. In particular, Dai and Terng gave an explicit construction of all solitons of the Ward equation by establishing a theory of Backlund transformation \cite{DT07}.

For the investigation of the Cauchy problem of the Ward equation, 
 Villarroel \cite{V90}, Dai, Terng and Uhlenbeck \cite{DTU} use Fourier analysis in the $x,y$-space to study the spectral theory of $\mathcal L_\lambda=\partial_y-\lambda\partial_x$ in (\ref{E:Lax01}), whilist Fokas and Ioannidou \cite{FI01} invert $\mathcal L_\lambda$ by interpreting it as a $1$-dimensional spectral operator with coefficients being the $x$-Fourier transform of functions.  In both cases, small data conditions of $Q$ are required to ensure the invertibility of $\mathcal L_\lambda$ and the solvability of the inverse problem.
Under the small data condition, the eigenfunctions $\Psi$ 
 possesses continuous scattering data only and therefore the solutions for the Ward equation do not include the solitons in previous study. 
 
%Besides, suppose the set of poles of $\Psi$ is nonempty and finite, Dai, Terng and Uhlenbeck \cite{DTU} apply the theory of Backlund transformation \cite{DT07} to solve the Cauchy problem of the Ward equation with mixed scattering data. Hence mathematicians are led to consider the Cauchy problem with purely continuous scattering data.

Nontheless, the approach of Fokas and Ioannidou \cite{FI01} shows that: after taking the Fourier transform in the $x$-space, (\ref{E:Lax01}) looks similar to the spectal problem of the AKNS system
\[(\partial_x-\lambda J)\Psi(x,t,\lambda)=q(x,t)\Psi(x,t,\lambda).\] 
Where  $J$ is a constant diagonal matrix with distinct eigenvalues. 
The solution of the forward and inverse scattering problem of the AKNS system is fairly complete, due to the work of Beals, Coifman, Deift, Tomei, Zhou \cite{BC84}, \cite{BDT88},  \cite{DZ91}. In particular, the inverse scattering problem for the AKNS system and its associated nonlinear evolution equations is rigorously solved for generic $q\in L_1$ without small data condition \cite{BC85}.

The purpose of the present paper is to remove the small data condition in solving the scattering and inverse scattering problem of (\ref{E:Lax01}) and the Cauchy problem of the Ward equation (\ref{E:chiral}) with a purely continuous scattering data.  We summarize principal results as follows:
\begin{theorem}\label{T:LNexistence}
Let $Q\in \mathbb P_{\infty,2,0}$. Then there is a bounded set $Z\subset \mathbb{C}$ such that 
\begin{itemize}
  \item $Z\cap\left(\mathbb{C}\backslash \mathbb{R}\right)$ is discrete in $\mathbb{C}\backslash \mathbb{R}$;
	\item For  $\lambda\in \mathbb{C}\backslash \left(\mathbb{R}\cup Z\right)$, the problem 
(\ref{E:Lax01}) has a unique solution $\Psi$ and $\Psi-1\in \mathbb {DH}^2$;
  \item For $(x,y)\in \mathbb{R}\times \mathbb{R}$, the eigenfunction $\Psi(x,y,\cdot)$ is meromorphic in $\lambda\in\mathbb{C}\backslash \mathbb{R}$ with poles precisely at the points of $Z\cap\left(\mathbb{C}\backslash \mathbb{R}\right)$;
  \item $\Psi(x,y,\lambda)$ satisfies:
  \begin{eqnarray}
&&\lim_{|x|\to\infty}\Psi(\cdot,y,\lambda)=1,\,\,
\lim_{|y|\to\infty}\Psi(x,\cdot,\lambda)=1,\,\,\textit{ for  $\lambda\in \mathbb{C}\backslash \left(\mathbb{R}\cup Z\right)$}, \label{E:bdry}\\
&&\textit{$\Psi(x,y,\cdot)$ tends to $1$ uniformly as $|\lambda|\to\infty$};\label{E:bdry''}
\end{eqnarray}
\item $\Psi(x,0,\lambda)$ satisfies:
  \begin{eqnarray}
&& \partial_x^i\left(\Psi-1\right), \,i=0,1,2,\textit{ are uniformly bounded in $L_2(dx)$ for} 
\label{E:bdry'}\\
&&\textit{$\lambda\in \mathbb{C}\backslash \left(\mathbb{R} \cup_{\lambda_j\in Z} D_{\epsilon_j}(\lambda_j)\right)$. For any $z_j\in\mathbb C\backslash\mathbb R$, fixing $\epsilon_k$ for $\forall k\ne j$    }\nonumber\\
&&\textit{and letting $\epsilon_j\to 0$, these $L_2(dx)$-norms increase as $ {C_j}{\epsilon_j^{-h_j}}$ with }\nonumber\\
&&\textit{uniform constants $C_j$, $h_j>0$.}\nonumber\\
%&& \textit{For $\forall \epsilon>0$, $(\Psi(x,0,\lambda)-1)\chi_{|x|>N}$, $(\partial_x\Psi(x,0,\lambda))\chi_{|x|>N}\to 0$ }
%\label{E:bdry'''}\\
%&&\textit{uniformly in $L_2(dx)$ as $N\to\infty$ for $\lambda\in \mathbb{C}\backslash \left(\mathbb{R} \cup_{\lambda_j\in Z} D_\epsilon(\lambda_j)\right)$.}\nonumber%\\
&& \textit{$\Psi-1,\,\partial_x\Psi\to 0$  in $L_2(dx)$ as $\lambda\to\infty$.}
\label{E:bdry'''}
\end{eqnarray}
  \end{itemize}
  Where $\epsilon\ge\epsilon_j>0$ are any given constants, $D_\epsilon(\lambda_j)$ denotes the disk of radius $\epsilon$ centered at $\lambda_j$.
\end{theorem}
Here the function spaces $\mathbb P_{\infty,2,0}$, and $\mathbb {DH}^2$ are defined by
\begin{definition}\label{D:L11}
\begin{eqnarray*}
&&\mathbb P_{\infty,k_1,k_2}\\
=&&\{q_x(x,y):\mathbb{R}\times\mathbb{R}\to su(n)|\,\\
 &&\begin{array}{lll}
 |\xi^iy^s\widehat q|_{L_1(d\xi dy)},&||\xi^h\widehat q(\xi,y)|_{L_1(y)}|_{ L_2(d\xi)},&{}\\
|\partial_x^j\partial_y^lq|_{L_\infty},&\sup_y|\partial_x^j\partial_y^lq|_{L_1(dx)},&|\partial_x^j\partial_y^lq|_{L_1(dxdy)}<\infty\end{array}\\
&&\textit{ for $1\le i\le \max\{5,k_1\},\,\,\,\, 0\le j,\, l\le \max\{5,k_1\},\,\,1\le h\le k_1$, $0\le s\le k_2$}\,\}.\\
&&\mathbb{DH}^k\\
=&&\{f\,| \partial_x^if(x,y)%,\in M_n(\mathbb C)\,\partial_y^if(x,y)
\textit{ are uniformly bounded in } L_2(\mathbb R, dx), \,\,0\le i\le k.\}
\end{eqnarray*}
 \end{definition}
%Here the derivatives are taken in distribution sense.
To derive Theorem \ref{T:LNexistence}, we transform the existence problem of $\Psi$ into a Riemann-Hilbert problem with a non-small continuous data 
by the translating invariant and the derivation properties of the spectral operator $\mathcal L_\lambda$, and an induction scheme.
 Hence the scheme of Section 10 in \cite{BC84} can be adapted to solve the Riemann-Hilbert problem. That is, we  
 first approximate the solution by a piecewise rational function. Then the correction is made by a solution of a Riemann-Hilbert problem with small data and a solution of a finite linear system.  
Since the eigenfunction obtained in each induction step consists the data of the Riemann-Hilbert problem in the next step, we need to obtain the $H^2$-estimate (\ref{E:bdry'}) of the eigenfunction. Besides, the boundary estimate (\ref{E:bdry'''}) and the meromorphic property are derived in each step to assure the solvability of the linear system.

In general, the points in $Z$, i.e., poles of $\Psi(x,y,\lambda)$, will occure or accumulate on the real line, or the limit points will accumulate themselves. Assuming higher regularities on the potential $Q$ and $Z=Z(\Psi)=\phi$ (there are no poles of $\Psi(x,y,\lambda)$), we can extract the continuous scattering data:
\begin{theorem}\label{T:CSDsum}
For $Q\in {\mathbb P}_{\infty,k,1}$, $k\ge 7$, if $Z=\phi$, then there exists uniquely a function $v(x, y, \lambda)\in \mathfrak{S}_{c,k}$ which satisfies 
\[
\Psi_{+}(x,y,\lambda)=\Psi_-(x,y,\lambda)v(x,y, \lambda), \qquad \lambda\in R.
\] 
\end{theorem}  

 Where the space $\mathfrak{S}_{c,k}$ is defined by
\begin{definition}\label{D:pcd}
Let $\mathfrak{S}_{c,k}$, $k\ge 7$, be the space consisting of continuous scattering data $v(x, y, \lambda)$,  $\lambda\in\mathbb R$, 
such that  $v$ satisfies the algebraic constraints:
\begin{eqnarray}
&&\det \,(v)\equiv 1, \label{E:real1'}\\
&& v= v^*>0,\label{E:real2'}
\end{eqnarray}
and the analytic  constraints: for $i+j\le k-4$, 
\begin{eqnarray}
&&\mathcal L_\lambda v=0,\,\,v(x,y,\lambda)=v(x+\lambda y, \lambda)\textit{ for }\forall x,\,y\in\mathbb R,\label{E:0ana15'}\\
&&\partial_x^i\partial_y^j\left(v-1\right)\textit{ are uniformly bounded in $ L_\infty\cap L_2(\mathbb{R},d\lambda)\cap L_1(\mathbb{R},d\lambda)$};\label{E:0anal2'}\\ 
		 		 %&&\textit{}\nonumber\\
&&\partial_x^i\partial_y^j\left(v-1\right)\to 0\textit{ uniformly in  $L_\infty\cap L_2(\mathbb{R},d\lambda)\cap L_1(\mathbb{R},d\lambda)$   }\label{E:0ana14'}\\
&&\textit{as $|x|$ or $|y|\to\infty$;}\nonumber\\
&&\partial_\lambda v\, \textsl{ are in $  L_2(\mathbb{R},d\lambda)$ and the norms depend continuously on $x$, $y$.}\label{E:0anal3'}
%Where $M$ is any given positive numbers. 
\end{eqnarray}
Where 
$
\mathcal L_\lambda=\partial_y-\lambda\partial_x$.
\end{definition}
The characterization of the scattering data $v\in \mathfrak{S}_{c,k}$ is necessary. Since the Cauchy integral operator  will play a key role in  the inverse problem. The study of the asymptotic behavior of the scattering data $v$  (hence the asymptotic behavior of the eigenfunctions $\Psi$) is important. Because the Cauchy operator is bounded in $L_2$ \cite{SW}, in general, an $L_2$-estimate of $\Psi$ and its derivatives will be good enough. However, a formal calculation will yield  (\ref{E:invq}) if the inverse problem is solvable. Hence we provide the estimates (\ref{E:0anal2'})-(\ref{E:0anal3'}). 

The derivation of (\ref{E:0anal2'})-(\ref{E:0anal3'}) basically relies on the $L_2$-boundness of the Cauchy operator and the estimates obtained in the small-data problem. In particular, both of the $1$-dimensional (Fokas and Ioannidou \cite{FI01} or (\ref{E:eigen+})) and the $2$-dimensional formulation (Villarroel \cite{V90} or (\ref{E:2d})) of the spectral problem are crucial in the derivation of the estimates with small data condition. That is, using (\ref{E:eigen+}), boundness or integrability in $x$-variable of the eigenfunctions $\Psi$ comes first from the differentiability and integrability of the potentials $Q$ via the Fourier transform. Then, strong asympote in $x$, $y$ or $\lambda$-variable of the eigenfunctions $\Psi$ can be obtained by (\ref{E:2d}) and previous estimates. We lose some regularities in deriving strong asymptote. See the proof of Theorem \ref{T:lambda2} for example.

For the inverse problem, our results are:
\begin{theorem}\label{T:invexistence}
Given $v(x, y, \lambda)\in \mathfrak{S}_{c,k}$, $k\ge 7$, 
there exists a unique solution $\Psi(x,y,\cdot)$ for the Riemann-Hilbert problem $(\lambda\in\mathbb{R}, v(x, y, \lambda))$ such that
\begin{equation}
\Psi-1,\, \partial_x\Psi,\,\partial_y\Psi \textit{ are uniformly bounded in $L_2(\mathbb R, d\lambda)$.} \label{E:invbdry''}
\end{equation}
Moreover, for each fixed $\lambda\notin\mathbb R$, and  $i+j\le k-4$,
\begin{gather}
\partial_x^i\partial_y^j\Psi\in L_\infty(dxdy),\label{E:invbdry'}\\
\partial_x^i\partial_y^j\left(\Psi-1\right)\to 0  \textit{ in $L_\infty(dxdy)$, as $x$ or $y\to \infty$}.\label{E:invbdry}
\end{gather}
\end{theorem}
Theorem \ref{T:invexistence} is proved by a Riemann-Hilbert problem with a non-small purely continuous  scattering data.  Without uniform boundedness of $\partial_\lambda v$, we need to handle seperately the Riemann-Hilbert problem for $|\lambda|>M>>1$ and $|\lambda|\le M$. For $|\lambda|>M>>1$, the Riemann-Hilbert problem is a small-data problem and hence can be solved. For $|\lambda|\le M$, the Riemann-Hilbert problem is again factorized into a diagonal problem, a Riemann-Hilbert problem with small data, and a finite linear system. Note we obtain the globally solvability by applying the Fredholm property and the reality condition (\ref{E:real2'}).

Moreover, good estimates for $\Psi$ can be derived only for $\lambda\notin \mathbb R$. However, it is enough to imply satisfactory analytical properties of the potentials.
\begin{theorem}\label{T:invcd}
Given $v(x, y, \lambda)\in \mathfrak{S}_{c,k}$, $k\ge 7$, the eigenfunction $\Psi$ obtained by Theorem \ref{T:invexistence}  satisfies (\ref{E:Lax01})  with 
\begin{equation}
Q(x,y)= \frac 1{2\pi i}\int_\mathbb{R}\Psi_-(v-1)d\zeta,\label{E:invq}
\end{equation}and $\Psi(x,\cdot,\lambda)\to 1$ as $y\to-\infty$. Where $\partial_xQ(x,y)\in su(n)$, and for $i+j\le k-4$, $i>0$, $\partial_{ x}^i\partial_{ y}^j Q$, $\partial_{ y} Q$, $Q\in L_\infty$, $\partial_{ x}^i\partial_{ y}^jQ$, $\partial_{ y}  Q$, $Q\to 0$ as $x$ or $y\to\infty$. 
\end{theorem}
Applying Theorem \ref{T:LNexistence}-\ref{T:invcd}, we extend the results of \cite{V90}, \cite{FI01}, \cite{DTU} by:
\begin{theorem}\label{T:cauchy}
If $Q_0\in \mathbb P_{\infty,k,1}$, $k\ge 7$, and there are no poles of the eigenfunction $\Psi_0$ of $Q_0$, then the Cauchy problem of the Ward equation (\ref{E:chiral}) with initial condition $Q(x,y,0)=Q_0(x,y)$
admits a smooth global solution satisfying: for $i+j+h\le k-4$, $i^2+j^2>0$,
\begin{eqnarray*}
&&\partial_x Q(x,y,t)\in su(n),\\
&&\partial_x^i\partial_y^j\partial_t^hQ,\,\partial_tQ,\,Q\in L_\infty,\\
&&\partial_x^i\partial_y^j\partial_t^hQ,\,\partial_tQ,\,Q\to 0, \,{\textit as }\,\,x,\,y,\,t\to\infty.
\end{eqnarray*}
 
\end{theorem}

The paper is organized as: In Section \ref{E:DPSMexistence}, we review an existence theorem of Fokas and Ioannidou \cite{FI01} by  an analytical treatment. In Section \ref{S:DPSMasymp}, under the small-data constraint, we analyze the asymptotic behavior of the eigenfunctions. In Section \ref{S:ENS} and \ref{S:CSD}, we solve the direct problem by justifying Theorem \ref{T:LNexistence} and \ref{T:CSDsum}. The inverse problem is complete in Section \ref{S:CSDinv} by proving Theorem \ref{T:invexistence} and \ref{T:invcd}. Finally, Theorem \ref{T:cauchy} is  proved in Section \ref{S:cauchy}.

\noindent{\textbf{Acknowledgements}}

I would like to express my gratitude to Richard Beals for the hospitality during my visit at Yale in the summer of 2006 and many helpful discussions during the preparation of this report. I would like to thank Chuu-Lian Terng and Karen Uhlenbeck for pointing out the problem of Ward equation and for the generous encouragement during this work. 
%OVillarroel \cite{V90}, Dai, Terng and Uhlenbeck \cite{DTU} use Fourier analysis in the $x,y$-space to study the spectral theory of $\mathcal L_\lambda=\partial_y-\lambda\partial_x$ in (\ref{E:Lax01}), whilist  invert $\mathcal L_\lambda$ by interpreting it as a $1$-dimensional spectral operator with coefficients being the $x$-Fourier transform of functions.  In both cases, small data condition
%%%%%%%%%%%%%%%%%%%%%%%%%%%%%%%%%%%%%%%%%%
\section{Direct problem I: Eigenfunctions with small data}\label{E:DPSMexistence}
%%%%%%%%%%%%%%%%%%%%%%%%%%%%%%%%%%%%%%%%%%%

Given a potential $\partial_x Q(x,y):\mathbb{R}\times \mathbb{R}\to su(n)$, and a constant $\lambda\in \mathbb{C}$, we consider the boundary value problem
\begin{eqnarray}
&&\partial_y\Psi(x,y,\lambda)-\lambda\partial_x\Psi(x,y,\lambda)-\left(\partial_xQ\right)\Psi(x,y,\lambda)=0, \label{E:Lax1}\\
&&\Psi(x,y,\lambda)\to 1,\qquad\textsl{as }y\to-\infty.\label{y-bdry}
\end{eqnarray}

To investigate the problem, we denote throughout that  

\begin{definition}\label{D:SMfunction}
\begin{eqnarray*}
&&\mathbb P_1=\left\{{\partial_xq}(x,y):\mathbb{R}\times\mathbb{R}\to  su(n)|\,|\xi\widehat q (\xi,y)|_{ L_1(d\xi dy)}<1\right\}, \\
&&\mathbb X=\left\{w(x,y):\mathbb{R}\times\mathbb{R}\to  M_n(\mathbb{C})|\,\sup_y|\,\widehat w(\xi,y)|_{L_1(d\xi)}<\infty \right\},
\\
&&\widehat {\mathbb X }=\left\{ f(\xi,y):\mathbb{R}\times\mathbb{R}\to  M_n(\mathbb{C})|\, \sup_y| f (\xi,y)|_{L_1(d\xi)}<\infty\right\}.
\end{eqnarray*}
%and
%\[
%
%\]
Where $\,\,\widehat{}\,\,$ is the Fourier transform with respect to the $x$-variable, $M_n(\mathbb{C})$ is the space of $n\times n$ matrices, and for $f\in M_n(\mathbb{C})$
\begin{eqnarray*}
&&|f|=\mathrm{trace } \left(f^*f\right)^\frac 12,\qquad f^*=\overline f^T,\\ 
&&|f(\xi,y)|_{L_1(d\xi)}=\int_\mathbb{R}|f(\xi,y)|d\xi.
\end{eqnarray*}
\end{definition}

\begin{theorem}\label{T:SMexistence}
Suppose $Q\in \mathbb P_1$. Then for all fixed $\lambda\in \mathbb{C}^\pm$, there is uniquely a solution $\Psi$ of (\ref{E:Lax1}) and (\ref{y-bdry}) such that $\Psi-1\in \mathbb X$.  Moreover, for $\lambda\in \mathbb{C}^\pm$,
\begin{equation}
\lim_{|x|\to\infty}\Psi(\cdot,y,\lambda)= I,\qquad\lim_{|y|\to\infty}\Psi(x,\cdot,\lambda)= I.\label{E:SMxyinfty}%,\,\lim_{|\lambda|\to\infty, Im\,\lambda\ne 0}\Psi(x,y,\cdot)= I.%,\,\lim_{|\lambda|\to\infty}\Psi(x,y,\cdot)= I.
\end{equation}
\end{theorem}

\begin{proof} Write $\Psi=1+W$. Then (\ref{E:Lax1}), (\ref{y-bdry}) are transformed into
\begin{eqnarray*}
&& \partial_yW-\lambda\partial_xW=\left({\partial_xQ}\right)W+\partial_xQ,\\
&& W(x,y,\lambda)\to 0 \textsl{ as } y\to-\infty.
\end{eqnarray*}
Taking the Fourier transform with respect to the $x$-variable (in distribution sense), we obtain
\[\partial_y\widehat W(\xi,y,\lambda)-i\xi\lambda\widehat W(\xi,y,\lambda)=\widehat{ \left(\partial_xQ\right)W}(\xi,y,\lambda)+\widehat {\partial_xQ}(\xi,y).
\]
Thus we are led to consider the following integral equations

\noindent

\begin{equation}
\widehat W(\xi,y,\lambda)=\begin{cases}
\int_{-\infty}^y\,e^{i\lambda\xi(y-y')}\left(\widehat {\partial_xQ}\ast{\widehat W}+\widehat {\partial_xQ}\right)\,dy',\textsl{ if  $\lambda\in \mathbb{C}^+$, $\xi\ge 0$};\\
-\int_y^\infty e^{i\lambda\xi(y-y')}\left(\widehat {\partial_xQ}\ast{\widehat W}+\widehat {\partial_xQ}\right)\,dy',\textsl{ if  $\lambda\in \mathbb{C}^+$, $\xi\le 0$};\\
-\int_y^\infty  \,e^{i\lambda\xi(y-y')}\left(\widehat{\partial_xQ}\ast{\widehat W}+\widehat {\partial_xQ}\right)\,dy',\textsl{ if  $\lambda\in \mathbb{C}^-$, $\xi\ge 0$};\\
\int_{-\infty}^y\, e^{i\lambda\xi(y-y')}\left(\widehat {\partial_xQ}\ast{\widehat W}+\widehat {\partial_xQ}\right)\,dy',\textsl{ if  $\lambda\in \mathbb{C}^-$, $\xi\le 0$}.\end{cases}
\label{E:W1}\end{equation}
Where $\ast$ is the convolution operator with respect to the $\xi$-variable. Define

%\begin{lemma}\label{L:hatX}
%If $Q_x\in P_0$, then the system (\ref{E:W1})-(\ref{E:W4}) is solved in the space $\widehat X$. 
%\end{lemma} 

%\begin{proof} We only prove the solvability of (\ref{E:W1}). The remaining cases can be done similarly. Let
\[
\mathcal K_\lambda f(\xi,y,\lambda)=\begin{cases}\int_{-\infty}^y\,e^{i\lambda\xi(y-y')}\left(\widehat {\partial_xQ}\ast f\right)(\xi,y',\lambda)\,dy', &\textsl{ if $\lambda\in\mathbb{C}^+$, $\xi\ge 0$;}\\
-\int_{y}^\infty\,e^{i\lambda\xi(y-y')}\left(\widehat {\partial_xQ}\ast f\right)(\xi,y',\lambda)\,dy', &\textsl{ if $\lambda\in\mathbb{C}^+$, $\xi\le 0$;}\\
-\int_{y}^\infty\,e^{i\lambda\xi(y-y')}\left(\widehat {\partial_xQ}\ast f\right)(\xi,y',\lambda)\,dy', &\textsl{ if $\lambda\in\mathbb{C}^-$, $\xi\ge 0$;}\\
\int_{-\infty}^y\,e^{i\lambda\xi(y-y')}\left(\widehat {\partial_xQ}\ast f\right)(\xi,y',\lambda)\,dy', &\textsl{ if $\lambda\in\mathbb{C}^-$, $\xi\le 0$.}
\end{cases}
\]Thus (\ref{E:W1}) turns into
\begin{equation}
\widehat W=\begin{cases}
\mathcal K_\lambda \widehat W+\int_{-\infty}^y e^{i\lambda\xi(y-y')}\widehat {\partial_xQ}(\xi,y')dy',&\textsl{ if $\lambda\in\mathbb{C}^+$, $\xi\ge 0$;}\\
\mathcal K_\lambda \widehat W-\int_{y}^\infty e^{i\lambda\xi(y-y')}\widehat {\partial_xQ}(\xi,y')dy',&\textsl{ if $\lambda\in\mathbb{C}^+$, $\xi\le 0$;}\\
\mathcal K_\lambda \widehat W-\int_{y}^\infty e^{i\lambda\xi(y-y')}\widehat {\partial_xQ}(\xi,y')dy',&\textsl{ if $\lambda\in\mathbb{C}^-$, $\xi\ge 0$;}\\
\mathcal K_\lambda \widehat W+\int_{-\infty}^y e^{i\lambda\xi(y-y')}\widehat {\partial_xQ}(\xi,y')dy',&\textsl{ if $\lambda\in\mathbb{C}^-$, $\xi\le 0$.}
\end{cases}\label{E:WK}
\end{equation}
where $\int_{-\infty}^y e^{i\lambda\xi(y-y')}\widehat {\partial_xQ}(\xi,y')dy',\,\, \int_{y}^\infty e^{i\lambda\xi(y-y')}\widehat {\partial_xQ}(\xi,y')dy'\in \widehat {\mathbb X}$ by $Q\in \mathbb P_1$. Note
\begin{eqnarray*}
%|Kf|_{L^\infty }&&\le \int_{-\infty}^y\,|\widehat {Q_x}(\xi,y')|_{L^1(\xi)} | f(\xi,y')|_{L^\infty(\xi)}dy'\\
%&&\le|\widehat {Q_x}(\xi,y)|_{L^1(\xi,y)}| f|_{L^\infty}\\
|\mathcal K_\lambda f(\xi,y)|_{L_1 (d\xi)}&&\le \int_{-\infty}^\infty\,|\widehat {\partial_xQ}(\xi,y')|_{L_1(d\xi)}|  f(\xi,y')|_{L_1(d\xi)}\,dy'\\
&&\le |\widehat {\partial_xQ}(\xi,y)|_{L_1(d\xi dy)}\,\sup_y| f|_{L_1(d\xi)}.
\end{eqnarray*}
Hence
\begin{equation}
\mathcal K_\lambda :\widehat {\mathbb X}\to \widehat {\mathbb X},\qquad
\Vert \mathcal K_\lambda \Vert\le  |\widehat {\partial_xQ}(\xi,y)|_{L_1(d\xi dy)}<1.\label{E:K}
\end{equation}
So 
\[\widehat W=\begin{cases}
(1-\mathcal K_\lambda)^{-1}\int_{-\infty}^y e^{i\lambda\xi(y-y')}\widehat {\partial_xQ}(\xi,y')dy',\qquad\textsl{ if $\lambda\in\mathbb{C}^+$, $\xi\ge 0$};\\ 
-(1-\mathcal K_\lambda)^{-1}\int_{y}^\infty e^{i\lambda\xi(y-y')}\widehat {\partial_xQ}(\xi,y')dy',\qquad\textsl{ if $\lambda\in\mathbb{C}^+$, $\xi\le 0$};\\
-(1-\mathcal  K_\lambda)^{-1}\int_y ^\infty e^{i\lambda\xi(y-y')}\widehat {\partial_xQ}(\xi,y')dy',\qquad\textsl{ if $\lambda\in\mathbb{C}^-$, $\xi\ge 0$};\\
(1-\mathcal K_\lambda)^{-1}\int_{-\infty}^y e^{i\lambda\xi(y-y')}\widehat {\partial_xQ}(\xi,y')dy',\qquad\textsl{ if $\lambda\in\mathbb{C}^-$, $\xi\le 0$}.
\end{cases}
\]Hence (\ref{E:W1}) is solvable if $Q\in \mathbb P_1$. Furthermore, the eigenfunction of (\ref{E:Lax1}), (\ref{y-bdry}) is given by:
\begin{equation}
\Psi(x,y,\lambda) =\begin{cases}
1+\frac 1{2\pi}\left(\int_0^\infty d\xi\int_{-\infty}^ydy'-\int_{-\infty}^0 d\xi\int_y^{\infty}dy'\right)\cdot\qquad\textsl{ if $\lambda\in \mathbb{C}^+$ }\\
 \hskip.06in e^{i\xi(x+\lambda(y-y'))}\left(\widehat {\partial_xQ}\ast{\widehat W}(\xi,y',\lambda)+\widehat {\partial_xQ}(\xi,y')\right);\\
1+\frac 1{2\pi}\left(\int_{-\infty}^0 d\xi\int_{-\infty}^ydy'- \int_0^\infty d\xi\int_y^{\infty}dy'\right)\cdot\qquad\textsl{ if $\lambda\in \mathbb{C}^-$ }\\
\hskip.06in  e^{i\xi(x+\lambda(y-y'))}\left(\widehat {\partial_xQ}\ast{\widehat W}(\xi,y',\lambda)+\widehat {\partial_xQ}(\xi,y')\right).\end{cases}\label{E:eigen+}\end{equation}

The uniqueness follows from  (\ref{E:Lax1}), (\ref{y-bdry}), (\ref{E:K}), the definition of $\mathbb X$,  and the contraction property of $\mathcal K_\lambda$.  

The uniform boundedness of $\Psi$ comes from Definition \ref{D:SMfunction}, (\ref{E:K}) and $Q\in \mathbb P_1$. By (\ref{E:eigen+}),  $\widehat {\partial_xQ}\ast{\widehat W}$, $\widehat {\partial_xQ}\in L_1(d\xi dy)$ and the Riemann-Lebesque Theorem, we obtain  $\Psi(\cdot,y, \lambda)\to 1$ as $|x|\to\infty$. On the other hand, (\ref{E:eigen+}), $\widehat {\partial_xQ}\ast{\widehat W}$, $\widehat {\partial_xQ}\in L_1(d\xi dy)$   and the Lebesque Convergence Theorem imply that $\Psi(x,\cdot, \lambda)\to 1$ when $|y|\to\infty$. 
\end{proof}
\begin{lemma}\label{L:SMdet}  
Suppose $\Psi$ satisfies (\ref{E:Lax1}), (\ref{y-bdry}). Then for $\lambda\notin\mathbb R$,
\[
\det\Psi(x,y,\lambda)\equiv 1.
\]
\end{lemma}

\begin{proof} Let $e_1, \,\cdots,\, e_n$ denote the standard basis for $\mathbb{C}^n$, $\psi_k$ the $k$-th column vector of the matrix $\Psi$.  Let $\Lambda^k(\mathbb{C}^n)$ denote the space of alternating $k$ forms on $\mathbb{C}^n$. Hence $\psi_1\wedge\psi_2\wedge\cdots\wedge\psi_n=\left(\det \Psi\right)\left(e_1\wedge e_2\wedge\cdots\wedge e_n\right)$. 
Taking derivatives of both sides, we derive
\begin{eqnarray*}
&&\left\{\left(\partial_y-\lambda\partial_x\right)\left(\det \Psi\right)\right\}\left(e_1\wedge e_2\wedge\cdots\wedge e_n\right)\\
=&&\left(\partial_y-\lambda\partial_x\right)\left\{\left(\det \Psi\right)\left(e_1\wedge e_2\wedge\cdots\wedge e_n\right)\right\}\\
=&&\left(\partial_y-\lambda\partial_x\right)\left\{\psi_1\wedge\psi_2\wedge\cdots\wedge\psi_n\right\}\\
=&&\left\{(\partial_y-\lambda\partial_x)\psi_1\right\}\wedge\cdots\wedge\psi_n+\cdots+\psi_1\wedge\cdots\wedge (\partial_y-\lambda\partial_x)\psi_n\\
=&&\left({\partial_xQ}\right)\psi_1\wedge\cdots\wedge\psi_n+\cdots+\psi_1\wedge\cdots \wedge \left({\partial_xQ}\right)\psi_n\\
=&& \left(\textsl{trace }{\partial_xQ}\right)\psi_1\wedge\psi_2\wedge\cdots\wedge\psi_n.
\end{eqnarray*}
So \[\left(\partial_y-\lambda\partial_x\right)\left(\det \Psi\right)=0\] by ${\partial_xQ}\in su(n)$. Moreover, for $\lambda\notin\mathbb R$, the equation turns into the debar equation
\[\partial_{\bar z}\left(\det \Psi\right)=0,\qquad  \textsl{$ x$, $ y\in \mathbb{R}$},
\]
by the change of variables:
\begin{equation}
x+\lambda y=\tilde x+i\tilde y= z,\qquad  \textsl{$\tilde x$, $\tilde y\in \mathbb{R}$}.\label{E:CV}
\end{equation}
Therefore the Liouville's Theorem and (\ref{E:SMxyinfty}) imply that $\det \Psi \equiv 1$, for $\lambda\notin\mathbb R$. %The lemma then follows from Lemma \ref{L:pmexistence} for $\forall \lambda\in \mathbb{C}$. 
\end{proof}

\begin{lemma}\label{L:SMreality}
 Suppose that $Q\in \mathbb P_1$. Then the reality condition 
\[
\Psi(x,y,\lambda)\Psi(x,y,\bar\lambda)^*=I
\]
holds for the eigenfunction $\Psi$. %Where $f^*=\overline{f}^T$.
\end{lemma}

\begin{proof} By Lemma \ref{L:SMdet}, one derives
\begin{eqnarray*}
&&(\partial_y-\lambda\partial_x){\Psi(x,y,\overline\lambda)^*}^{-1}\\
=&&-{\Psi(x,y,\overline\lambda)^*}^{-1}\left((\partial_y-\lambda\partial_x)\Psi(x,y,\overline\lambda)^*\right){\Psi(x,y,\overline\lambda)^*}^{-1}\\
=&&-{\Psi(x,y,\overline\lambda)^*}^{-1}\left(\overline{(\partial_y-\overline\lambda\partial_x)\Psi(x,y,\overline\lambda)}^T\right){\Psi(x,y,\overline\lambda)^*}^{-1}\\
=&&-{\Psi(x,y,\overline\lambda)^*}^{-1}(\left({\partial_xQ}\right)\Psi(x,y,\overline\lambda))^*{\Psi(x,y,\overline\lambda)^*}^{-1}\\
=&&-\left({\partial_xQ}^*\right){\Psi(x,y,\overline\lambda)^*}^{-1}\\
=&&\left({\partial_xQ}\right){\Psi(x,y,\overline\lambda)^*}^{-1}.
\end{eqnarray*}
Besides, noting $|\widehat {f^n}|_{L_1(d\xi)}\le |\widehat {f}|^n_{L_1(d\xi)}$ and the boundary condition of $\Psi$, we obtain $\Psi^{-1}-1\in \mathbb X$. Hence the lemma follows from the uniqueness property in Theorem \ref{T:SMexistence}. 
\end{proof}

%%%%%%%%%%%%%%%%%%%%%%%%%%%%%%%%%%%%%%%%%%
\section{Direct problem II: Asymptotic analysis with small data}\label{S:DPSMasymp}
%%%%%%%%%%%%%%%%%%%%%%%%%%%%%%%%%%%%%%%%%%%
The results and arguments will be applied or adapted in Section \ref{S:ENS} and \ref{S:CSD}. %In the case of Ward equation, the reality condition (see Lemma \ref{L:SMreality}) plays an important role in estimate.

Denote
\begin{eqnarray*}
\left(f\ast_{x,y}g\right)(x,y) &=&\int_{-\infty}^\infty \int_{-\infty}^\infty f(x-x',y-y')g(x',y')dx' dy',\\
\left(f\ast_{z,\bar z}g\right)(z,\bar z) &=&\int \int_{\mathbb{C}} f(z-\zeta,\bar z-\bar\zeta)g(\zeta,\bar\zeta)d\zeta d\bar\zeta.
\end{eqnarray*}
By the change of variables (\ref{E:CV}), we then have
\[\left(\partial_y-\lambda\partial_x\right)^{-1}=\frac i{2 \lambda_I}\partial_{\bar z}^{-1}=-\frac 1{4\pi \lambda_I z}\,\ast_{z,\bar z}=-\frac 1{2\pi i}\frac {\text {sgn}( \lambda_I)}{x+\lambda y}\ast_{x,y}
\] 
with $\lambda=\lambda_R+i\lambda_I$. Now let $\mathcal{S}$ be the set of Schwartz functions. If $Q\in \mathbb P_1\cap \mathcal{S}$, then the eigenfunction $\Psi$ obtained by Theorem \ref{T:SMexistence} satisfies
\begin{equation}
\Psi=1+ G_\lambda (\left({\partial_xQ}\right)\Psi)\label{E:2d}
\end{equation}
 Where
\begin{equation}
G_\lambda f(x,y,\lambda)=-\frac 1{2\pi i}\int_{-\infty}^\infty \int_{-\infty}^\infty \frac {\text {sgn}( \lambda_I)f(x-x',y-y',\lambda)}{{x'}+\lambda y'}dx'dy'\label{E:kernel}%=\frac 1{2\pi i}K_\lambda \left(Q_xf\right),
\end{equation}

%\begin{definition}\label{D:notation}
%For simplicity, we will use $C_{x,y}$ as a generic  bounded constant without further comment throughout Section \ref{S:DPSMasymp}. 
%\end{definition}

The following lemma is due to Richard Beals.
\begin{lemma}\label{L:FSestimate} 
Suppose $\varphi\in \mathcal{S}$. For $|\lambda|\ne 0$ and $| \lambda_I|<1$, %there exists a constant $C$ such that
\[
|G_\lambda \varphi|\le  \frac {C}{|\lambda|}(\sup_y|\partial_y\varphi|_{L_1(dx)}+\sup_y|\varphi|_{L_1(dx)}+|\varphi|_{L_1(dxdy)}). 
\]
Where $C$ is a constant. 
\end{lemma}

\begin{proof} Let $\frac 1{s}=\frac { \lambda_R}{ \lambda_R^2+ \lambda_I ^2}$. So 
\begin{equation}\frac 1\lambda=\frac 1{s}-i\frac { \lambda_I }{ \lambda_R}\frac 1s, \qquad |\frac 1{y+\frac x\lambda}|\le\frac 1{|y+\frac x{s}|}.\label{E:seps}
\end{equation}
Write 
\begin{eqnarray}
&&G_\lambda \varphi\nonumber\\
=&&\frac {-1}{2\pi i\lambda}(\int\int_{|y'+\frac {x'}{s}|<1}\text{sgn}{( \lambda_I)}\frac{\varphi(x-x',y-y')-\varphi(x-x',y+\frac{x'}{s})}{y'+\frac {x'}\lambda}dx'dy'\nonumber\\
&&+\int\int_{|y'+\frac {x'}{s}|>1}\text{sgn}{( \lambda_I)}\frac{\varphi(x-x',y-y')}{y'+\frac {x'}\lambda}dx'dy'\nonumber\\
&&+\int\int_{|y'+\frac {x'}{s}|<1}\text{sgn}{( \lambda_I)}\frac{\varphi(x-x',y+\frac{x'}{s})}{y'+\frac {x'}\lambda}dx'dy')\nonumber\\
=&&I_1+I_2+I_3.\nonumber
\end{eqnarray}
In view of (\ref{E:seps}), it is easy to see that
\begin{eqnarray}
|I_1|\le &&\frac 1{2\pi |\lambda|}\int\sup_{z: {|z+\frac {x'}{s}|<1}}|\partial_y\varphi(x-x',y-z)|dx'\nonumber\\
\le && \frac{C_1}{|\lambda|}\sup_y|\partial_y\varphi|_{L_1(dx)},\label{E:C1}\\
|I_2|\le &&\frac 1{2\pi |\lambda|} \int\int_{|y'+\frac {x'}{s}|>1}|\frac{\varphi(x-x',y-y')}{y'+\frac {x'}\lambda}|dx'dy'
\le  \frac {C_2}{|\lambda|}|\varphi|_{L_1(dxdy)}.\label{E:C2}
\end{eqnarray}
Finally, 
\begin{eqnarray*}
&&|\text{sgn}{( \lambda_I)}\int_{|y'+\frac {x'}{s}|<1}\frac 1{y'+\frac {x'}\lambda}dy'|\\
=&&|\log\left[\frac{1-i\frac{ \lambda_I x'}{ \lambda_R s}}{-1-i\frac{ \lambda_I x'}{ \lambda_R s}}\right]|\\
=&&|i\left[\text{arg}(1-i\frac{ \lambda_I x'}{ \lambda_Rs})-\text{arg}(-1-i\frac{ \lambda_I x'}{ \lambda_R s})\right]|\le\pi%\\
%=&&\text{sgn}{(x')}\pi i,
\end{eqnarray*}
%as $|\lambda|\to\infty$. 
This yields
\begin{eqnarray}
I_3&\le & \frac 1{2|\lambda|} \int| \varphi(x-x',y+\frac {x'}{s})|\,dx'\nonumber\\
&\le &\frac {C_3}{|\lambda|} \sup_y|\varphi|_{L_1(dx)}.\label{E:C3}
\end{eqnarray}
%One can verify $\lim_{|x|\to \infty}C_i=\lim_{|y|\to \infty}C_i=$ for $i=1,2,3$ by the Lebesque Convergence Theorem. 
Combining (\ref{E:C1}), (\ref{E:C2}), and (\ref{E:C3}), we prove the lemma.
\end{proof}

\begin{lemma}\label{L:asym-1}
 Suppose that $Q\in \mathbb P_1 \cap \mathcal{S}$. Then there exist a constant $C_N$ such that
\begin{eqnarray*}
&&|\partial_x^{N}\Psi|\le C_N.%,\\
%&&|\partial_x^N\partial_y\Psi|\le C_N|\lambda||\xi^N\widehat { Q_x}|_{L_1(d\xi dy)}.
\end{eqnarray*}
Where $C_N$ is a constant depending on $Q$.
\end{lemma}

\begin{proof} %By (\ref{E:eigen+}) and (\ref{E:Lax1}), it is sufficiently to show that 
%$ \xi^n\widehat{Q_x}\ast\widehat W\in L_1(d\xi dy)$. 
Since 
\begin{eqnarray*}
\xi^N\widehat{{\partial_xQ}}\ast\widehat W%=&&(\left(\xi-\xi'\right)+\xi')^N\widehat{Q_x}\ast\widehat W\\
=&&\sum_{k=0}^N\left(\begin{array}{l}N\\k\end{array}\right)\left(\xi-\xi'\right)^k\xi'^{N-k}\widehat{\partial_xQ}\ast\widehat W\\
=&&\sum_{k=0}^N\left(\begin{array}{l}N\\k\end{array}\right)\left(\xi^k\widehat{\partial_xQ}\right)\ast\left(\xi^{N-k}\widehat W\right).
\end{eqnarray*}
It suffices to prove $\xi^i\widehat W\in \mathbb X$ for $0\le k\le N$. This can be proved by induction on $k$ and  using the same argument as in the proof of Theorem \ref{T:SMexistence} if $|\xi^N\widehat {\partial_xQ}|_{L_1(d\xi dy)}<\infty$.
\end{proof} 

\begin{definition}\label{E:SMsc}
Define 
\begin{eqnarray*}
\mathbb{P}_{1,k}=&&\{{\partial_xq}(x,y):\mathbb{R}\times\mathbb{R}\to  su(n)|\,|\xi\widehat q (\xi,y)|_{ L_1(d\xi dy)}<1,\textit{ and}\\
&&\,\,\,|\xi^i\widehat q|_{L_1(d\xi dy)},\,\,|\partial_x^j\partial_y^hq|_{L_\infty},\,\,\sup_y|\partial_x^j\partial_y^hq|_{L_1(dx)},\,\,|\partial_x^j\partial_y^hq|_{L_1(dxdy)}<\infty\\
&&\,\,\,\textit{for $1\le i\le \max\{5,k\}$, $0\le j,\, h\le \max\{5,k\}$}.\}
\end{eqnarray*}
\end{definition}

Note that $\mathbb P_1\in\mathbb{P}_{1,k}$. For simplicity we abuse the notation $\partial_x^i \partial_y^j Q$, $\partial_x^i \partial_y^j \Psi$ by $Q_{\underbrace{x\cdots x}_{i} \underbrace{y\cdots y}_{j}}$, and $\Psi_{\underbrace{x\cdots x}_{i} \underbrace{y\cdots y}_{j}}$ in the remaining part of this section.

\begin{lemma}\label{L:asym-2}
Suppose that $Q\in \mathbb P_{1,k}$, $k\le 5$. Then \[\textit{$|\partial_x^{N}\Psi|\le C_N$, $0\le N\le 4$.}\]Moreover, as $|\lambda|\to\infty$,
\begin{eqnarray*}
&&|\partial_x\Psi|,\,\,|\partial_x^2\Psi|,\,\,|\partial_x^3\Psi|\le \frac { C}{|\lambda|}.%\left(\sum_{i=0}^2|\xi^i\widehat {Q_x}|_{L_1(d\xi dy)}\right)
%&&|\Psi_y|\le C.%\left(\sum_{i=0}^2|\xi^i\widehat {Q_x}|_{L_1(d\xi dy)}\right).
\end{eqnarray*}
Where $C_N$, $C$ is a constant depending on $Q$.
\end{lemma}

\begin{proof} The uniform boundedness of $\partial_x^N\Psi$, $0\le N\le 4$ in Lemma \ref{L:asym-1} will be used in the proof. A direct computation yields
\begin{eqnarray}
&&\Psi-(1-\frac Q\lambda)=\Psi\left(1-\Psi^{-1}(1-\frac Q\lambda)\right),\label{E:basic}\\
&&(\partial_y-\lambda\partial_x)\left(\Psi^{-1}(1-\frac Q\lambda)\right)
=-\frac 1{\lambda}\Psi^{-1}\left(Q_y-Q_xQ\right).\label{E:psix1}
\end{eqnarray}
So 
\[
\Psi_x+\frac {Q_x}\lambda=\Psi_x\left(1-\Psi^{-1}(1-\frac Q\lambda)\right)-\Psi\left(\Psi^{-1}(1-\frac Q\lambda)\right)_x= I_1+I_2\]
by (\ref{E:basic}).  
Therefore, inverting the operator $\partial_y-\lambda\partial_x$ in (\ref{E:psix1}) and applying Lemma \ref{L:FSestimate}, \ref{L:asym-1},
we have
\begin{eqnarray*}
|I_1|=&& \frac 1{|\lambda|}|\Psi_x||G_\lambda(\Psi^{-1}\left(Q_y-Q_xQ\right))|\\
\le &&  \frac {C}{|\lambda|^2}|\xi\widehat {Q_x}|_{L_1(d\xi dy)}(\sup_y|\left(\Psi^{-1}(Q_y-Q_xQ)\right)_y|_{L_1(dx)}
\\&&+\sup_y|\Psi^{-1}(Q_y-Q_xQ)|_{L_1(dx)}+|\Psi^{-1}\left(Q_y-Q_xQ\right)|_{L_1(dxdy)})\\
\le &&  \frac {C}{|\lambda|^2}\sum_{i=0}^1|\xi^i\widehat {Q_x}|_{L_1(d\xi dy)}(\left(|\Psi_y|_{L_\infty}+1\right)\sup_y|Q_y-Q_xQ|_{L_1(dx)}\\
&&+\sup_y|\left(Q_y-Q_xQ\right)_y|_{L_1(dx)}+|Q_y-Q_xQ|_{L_1(dxdy)})\\
\le &&  \frac {C}{|\lambda|^2}\sum_{i=0}^1|\xi\widehat {Q_x}|_{L_1(d\xi dy)}(\left(|\lambda\Psi_x+Q_x\Psi|_{L_\infty}+1\right)\sup_y|Q_y-Q_xQ|_{L_1(dx)}\\
&&+\sup_y|\left(Q_y-Q_xQ\right)_y|_{L_1(dx)}+|Q_y-Q_xQ|_{L_1(dxdy)})
\\
\le && \frac { C}{|\lambda|}\left(\sum_{i=0}^1|\xi^i\widehat {Q_x}|^2_{L_1(d\xi dy)}\right)\sum_{j,k=0}^2\left[\sup_y|\partial_x^j\partial_y^kQ|^2_{L_1(dx)}+|\partial_x^j\partial_y^kQ|^2_{L_1(dxdy)}\right]\\ 
\le && \frac { C}{|\lambda|}%\left(\sum_{i=0}^2|\xi^i\widehat {Q_x}|_{L_1(d\xi dy)}\right)
\end{eqnarray*}
as $|\lambda|\to\infty$.  Taking the $x$-derivatives of both sides of (\ref{E:psix1}), we derive 
\begin{eqnarray*}
|I_2|
=&&  \frac 1{|\lambda|}|\Psi G_\lambda\left(\Psi^{-1}(Q_y-Q_xQ)\right)_x|\\
\le &&  \frac {C}{|\lambda|^2}(\sup_y|\left(\Psi^{-1}(Q_y-Q_xQ)\right)_{xy}|_{L_1(dx)}\\
&&+\sup_y|\left(\Psi^{-1}(Q_y-Q_xQ)\right)_{x}|_{L_1(dx)}+|\left(\Psi^{-1}(Q_y-Q_xQ)\right)_x|_{L_1(dxdy)})\\
\le && \frac { C}{|\lambda|}\left(\sum_{i=0}^2|\xi^i\widehat {Q_x}|^2_{L_1(d\xi dy)}\right)\sum_{j,k=0}^3\left[\sup_y|\partial_x^j\partial_y^kQ|^2_{L_1(dx)}+|\partial_x^j\partial_y^kQ|^2_{L_1(dxdy)}\right]\\ 
\le && \frac { C}{|\lambda|}%\left(\sum_{i=0}^2|\xi^i\widehat {Q_x}|_{L_1(d\xi dy)}\right)
\end{eqnarray*}
Here we have used that (\ref{E:Lax1}) and Lemma \ref{L:asym-1}. 

By the same scheme as above and the following equalities
\begin{eqnarray*}
\Psi_{xx}+\frac { Q_{xx}}\lambda
=&&
\Psi_{xx}\left(1-\Psi^{-1}(1-\frac Q\lambda)\right)+
2\Psi_x\left(1-\Psi^{-1}(1-\frac Q\lambda)\right)_x\\
&&+\Psi\left(1-\Psi^{-1}(1-\frac Q\lambda)\right)_{xx}\\
\Psi_{xxx}+\frac { Q_{xxx}}\lambda
=&&
\Psi_{xxx}\left(1-\Psi^{-1}(1-\frac Q\lambda)\right)+
3\Psi_{xx}\left(1-\Psi^{-1}(1-\frac Q\lambda)\right)_x\\
&&+
3\Psi_x\left(1-\Psi^{-1}(1-\frac Q\lambda)\right)_{xx}+\Psi\left(1-\Psi^{-1}(1-\frac Q\lambda)\right)_{xxx},
\end{eqnarray*}
one derives
\begin{eqnarray*}
|\Psi_{xx}|&&\le \frac { C}{|\lambda|}\left(\sum_{i=0}^3|\xi^i\widehat {Q_x}|^2_{L_1(d\xi dy)}\right)\sum_{j,k=0}^4\left[\sup_y|\partial_x^j\partial_y^kQ|^2_{L_1(dx)}+|\partial_x^j\partial_y^kQ|^2_{L_1(dxdy)}\right]\\
%&&\le \frac { C}{|\lambda|}\\
|\Psi_{xxx}|&&\le \frac { C}{|\lambda|}\left(\sum_{i=0}^4|\xi^i\widehat {Q_x}|^2_{L_1(d\xi dy)}\right)\sum_{j,k=0}^5\left[\sup_y|\partial_x^j\partial_y^kQ|^2_{L_1(dx)}+|\partial_x^j\partial_y^kQ|^2_{L_1(dxdy)}\right]
%&&\le \frac { C}{|\lambda|}.
\end{eqnarray*}
Hence the estimates for $\Psi_{xx}$ and $\Psi_{xxx}$ follow.
\end{proof}

\begin{lemma}\label{L:asym-3}
Suppose that $Q\in \mathbb P_{1,k}$, $k\le 5$. Then 
\begin{eqnarray}
&&|\partial_y\Psi|\le \frac { C}{|\lambda|},\label{E:py}\\
&&|\partial_x\partial_y\Psi|\le \frac { C}{|\lambda|},\label{E:pxpy}
\end{eqnarray}
as $|\lambda|\to\infty$. Where $C$ is a constant depending on $Q$.
\end{lemma}

\begin{proof}
Using the formula
\[
\Psi_y+\frac {Q_y}\lambda=\Psi_y\left(1-\Psi^{-1}(1-\frac Q\lambda)\right)-\Psi\left(\Psi^{-1}(1-\frac Q\lambda)\right)_y= II_1+II_2.\]
and Lemma \ref{L:asym-2}, one can derive
\begin{eqnarray*}
&&|II_1|\\
=&&  \frac 1{|\lambda|}|\Psi_y G_\lambda\left(\Psi^{-1}(Q_y-Q_xQ)\right)|\\
\le &&  \frac {C}{|\lambda|^2}|\Psi_y|(\sup_y|\left(\Psi^{-1}(Q_y-Q_xQ)\right)_{y}|_{L_1(dx)}+\sup_y|\Psi^{-1}(Q_y-Q_xQ)|_{L_1(dx)}
\\
&&+|\Psi^{-1}(Q_y-Q_xQ)|_{L_1(dxdy)})
\\
\le && \frac { C}{|\lambda|^2}\left(\sum_{i=0}^2|\xi^i\widehat {Q_x}|^2_{L_1(d\xi dy)}\right)\sum_{j,k=0}^3\left[\sup_y|\partial_x^j\partial_y^kQ|^2_{L_1(dx)}+|\partial_x^j\partial_y^kQ|^2_{L_1(dxdy)}\right]
\\
&&\textit{(by estimates of $I_1$, and $I_2$ in Lemma \ref{L:asym-2})}\\
\le && \frac { C}{|\lambda|^2},%\left(\sum_{i=0}^2|\xi^i\widehat {Q_x}|_{L_1(d\xi dy)}\right),
\end{eqnarray*}
\begin{eqnarray*}
&&|II_2|\\
=&&  \frac 1{|\lambda|}|\Psi G_\lambda\left(\Psi^{-1}(Q_y-Q_xQ)\right)_y|
\\
\le &&  \frac {C}{|\lambda|^2}(\sup_y|\left(\Psi^{-1}(Q_y-Q_xQ)\right)_{yy}|_{L_1(dx)}+\sup_y|\left(\Psi^{-1}(Q_y-Q_xQ)\right)_y|_{L_1(dx)}
\\
&&+|\left(\Psi^{-1}(Q_y-Q_xQ)_y\right)|_{L_1(dxdy)})
\\
\le && \frac { C}{|\lambda|}\left(\sum_{i=0}^3|\xi^i\widehat {Q_x}|^2_{L_1(d\xi dy)}\right)\sum_{j,k=0}^4\left[\sup_y|\partial_x^j\partial_y^kQ|^2_{L_1(dx)}+|\partial_x^j\partial_y^kQ|^2_{L_1(dxdy)}\right]\\ 
&&\textit{(by estimates of $\Psi_{xx}$ in Lemma \ref{L:asym-2})}\\
\le &&\frac { C}{|\lambda|}.%\left(\sum_{i=0}^2|\xi^i\widehat {Q_x}|_{L_1(d\xi dy)}\right).
\end{eqnarray*}
Where the estimate $|\Psi_{yy}|=|\lambda^2\Psi_{xx}+\lambda(Q_x\Psi)_x+(Q_x\Psi)_y|$ has been used. Thus (\ref{E:py}) is proved. 
On the other hand, we write
\begin{eqnarray*}
&&\Psi_{xy}+\frac { Q_{xy}}\lambda\\
=&&
\Psi_{xy}\left(1-\Psi^{-1}(1-\frac Q\lambda)\right)-
\Psi_x\left(1-\Psi^{-1}(1-\frac Q\lambda)\right)_y\\
&&+
\Psi_y\left(1-\Psi^{-1}(1-\frac Q\lambda)\right)_x-\Psi\left(1-\Psi^{-1}(1-\frac Q\lambda)\right)_{xy}\\
=&&III_1+III_2+III_3+III_4.
\end{eqnarray*}
Similarly, one can verify 
\begin{eqnarray*}
|III_1|\le &&\frac { C}{|\lambda|^2}\sum_{i=0}^3|\xi^i\widehat {Q_x}|^2_{L_1(d\xi dy)}\sum_{j,k=0}^4\left[\sup_y|\partial_x^j\partial_y^kQ|^2_{L_1(dx)}+|\partial_x^j\partial_y^kQ|^2_{L_1(dxdy)}\right],
\\
|III_2|\le && \frac { C}{|\lambda|^2}\sum_{i=0}^3|\xi^i\widehat {Q_x}|^2_{L_1(d\xi dy)}\sum_{j,k=0}^4\left[\sup_y|\partial_x^j\partial_y^kQ|^2_{L_1(dx)}+|\partial_x^j\partial_y^kQ|^2_{L_1(dxdy)}\right],
\\
|III_3|\le &&\frac { C}{|\lambda|^3}\sum_{i=0}^3|\xi^i\widehat {Q_x}|^2_{L_1(d\xi dy)}\sum_{j,k=0}^4\left[\sup_y|\partial_x^j\partial_y^kQ|^2_{L_1(dx)}+|\partial_x^j\partial_y^kQ|^2_{L_1(dxdy)}\right],\\
|III_4|\le &&\frac { C}{|\lambda|}\sum_{i=0}^4|\xi^i\widehat {Q_x}|^2_{L_1(d\xi dy)}\sum_{j,k=0}^5\left[\sup_y|\partial_x^j\partial_y^kQ|^2_{L_1(dx)}+|\partial_x^j\partial_y^kQ|^2_{L_1(dxdy)}\right],
\end{eqnarray*}
by Lemma \ref{L:asym-2}, (\ref{E:py}). Hence we prove (\ref{E:pxpy}).
\end{proof}

\begin{theorem}\label{T:lambda2}
If $Q\in \mathbb P_{1,k}$, $k\le 5$, then as $|\lambda|\to\infty$,
\begin{eqnarray}
&&|\Psi(x,y,\lambda)-\left(1-\frac Q\lambda \right)|\le \frac { C}{|\lambda|^2},\label{E:psi-asym}\\
&&|\partial_x\Psi(x,y,\lambda)+\frac {\partial_xQ}\lambda |,\,\,|\partial_y\Psi(x,y,\lambda)+\frac {\partial_yQ}\lambda |\le \frac { C}{|\lambda|^2}.\label{E:psixy-asym}
\end{eqnarray}
Where $C$ is a constant depending on $Q$. \end{theorem}

\begin{proof}  Applying (\ref{E:psix1}), Lemma \ref{L:asym-2}, and \ref{L:asym-3}, we obtain
\begin{eqnarray*}
&&|\Psi-\left(1-\frac Q\lambda\right)|\nonumber\\
=&&|\Psi|\,|1-\Psi^{-1}\left(1-\frac Q\lambda\right)|\nonumber\\
=&&\frac {|\Psi|} {|\lambda|}|G_\lambda \left(\Psi^{-1}\left(Q_y-Q_xQ\right)\right)|\nonumber\\
\le && \frac {C}{|\lambda|^2}|\widehat {Q_x}|_{L_1(d\xi dy)} (|\sup_y\left(\Psi^{-1}\left(Q_y-Q_xQ\right)\right)_y|_{L_1(dx)}
\\
&&+|\sup_y \Psi^{-1}\left(Q_y-Q_xQ\right)|_{L_1(dx)}+|\Psi^{-1}\left(Q_y-Q_xQ\right)|_{L_1(dxdy)})\nonumber\\
\le &&  \frac {C}{|\lambda|^2}\sum_{i=0}^3|\xi^i\widehat {Q_x}|^2_{L_1(d\xi dy)}\sum_{j,k=0}^4\left[\sup_y|\partial_x^j\partial_y^kQ|^2_{L_1(dx)}+|\partial_x^j\partial_y^kQ|^2_{L_1(dxdy)}\right]
\end{eqnarray*} 
as $|\lambda|\to\infty$. Therefore, (\ref{E:psi-asym}) is proved. 

To proved (\ref{E:psixy-asym}), we used the results of Lemma \ref{L:asym-2}, and \ref{L:asym-3} to improve the estimates of $I_1$, $I_2$,  $II_1$, and $II_2$ in the proof of Lemma \ref{L:asym-2}, \ref{L:asym-3}. More precisely,
\begin{eqnarray*}
|I_1|
=&& \frac 1{|\lambda|}|\Psi_x||G_\lambda(\Psi^{-1}\left(Q_y-Q_xQ\right))|\\
\le &&  \frac {C}{|\lambda|^2}|\Psi_x|(\sup_y|\left(\Psi^{-1}(Q_y-Q_xQ)\right)_y|_{L_1(dx)}
\\
&&+\sup_y|\Psi^{-1}(Q_y-Q_xQ)|_{L_1(dx)}+|\Psi^{-1}\left(Q_y-Q_xQ\right)|_{L_1(dxdy)})\\
\le &&  \frac {C}{|\lambda|^3}\sum_{i=0}^3|\xi^i\widehat {Q_x}|^2_{L_1(d\xi dy)}\sum_{j,k=0}^4\left[\sup_y|\partial_x^j\partial_y^kQ|^2_{L_1(dx)}+|\partial_x^j\partial_y^kQ|^2_{L_1(dxdy)}\right],
\\|I_2|
=&&  \frac 1{|\lambda|}|\Psi G_\lambda\left(\Psi^{-1}(Q_y-Q_xQ)\right)_x|\\
\le &&  \frac {C}{|\lambda|^2}(\sup_y|\left(\Psi^{-1}(Q_y-Q_xQ)\right)_{xy}|_{L_1(dx)}\\
&&+\sup_y|\left(\Psi^{-1}(Q_y-Q_xQ)\right)_{x}|_{L_1(dx)}+|\left(\Psi^{-1}(Q_y-Q_xQ)\right)_x|_{L_1(dxdy)})\\
\le &&  \frac {C}{|\lambda|^2}\sum_{i=0}^3|\xi^i\widehat {Q_x}|^2_{L_1(d\xi dy)}\sum_{j,k=0}^4\left[\sup_y|\partial_x^j\partial_y^kQ|^2_{L_1(dx)}+|\partial_x^j\partial_y^kQ|^2_{L_1(dxdy)}\right],
\\
|II_1|
=&&  \frac 1{|\lambda|}|\Psi_y G_\lambda\left(\Psi^{-1}(Q_y-Q_xQ)\right)|\\
\le &&  \frac {C}{|\lambda|^2}|\Psi_y|(\sup_y|\left(\Psi^{-1}(Q_y-Q_xQ)\right)_{y}|_{L_1(dx)}\\
&&+\sup_y|\Psi^{-1}(Q_y-Q_xQ)|_{L_1(dx)}+|\Psi^{-1}(Q_y-Q_xQ)|_{L_1(dxdy)})\\
\le &&  \frac {C}{|\lambda|^3}\sum_{i=0}^3|\xi^i\widehat {Q_x}|^2_{L_1(d\xi dy)}\sum_{j,k=0}^4\left[\sup_y|\partial_x^j\partial_y^kQ|^2_{L_1(dx)}+|\partial_x^j\partial_y^kQ|^2_{L_1(dxdy)}\right],
\end{eqnarray*}
\begin{eqnarray*}
|II_2|
=&&  \frac 1{|\lambda|}|\Psi G_\lambda\left(\Psi^{-1}(Q_y-Q_xQ)\right)_y|\\
\le &&  \frac {C}{|\lambda|^2}(\sup_y|\left(\Psi^{-1}(Q_y-Q_xQ)\right)_{yy}|_{L_1(dx)}
\\
&&+\sup_y|\left(\Psi^{-1}(Q_y-Q_xQ)\right)_y|_{L_1(dx)}+|\left(\Psi^{-1}(Q_y-Q_xQ)_y\right)|_{L_1(dxdy)})\\
\le &&  \frac {C}{|\lambda|^2}\sum_{i=0}^4|\xi^i\widehat {Q_x}|^2_{L_1(d\xi dy)}\sum_{j,k=0}^5\left[\sup_y|\partial_x^j\partial_y^kQ|^2_{L_1(dx)}+|\partial_x^j\partial_y^kQ|^2_{L_1(dxdy)}\right].\end{eqnarray*}
Here $|\Psi_{yy}|=|\lambda\Psi_{xy}+Q_{xy}\Psi+Q_x\Psi_y|$ and (\ref{E:pxpy}) have been used in the estimation of $II_2$.  
\end{proof}

By induction, we can generalize the results of Lemma \ref{L:asym-1}-\ref{L:asym-3} and Theorem \ref{T:lambda2} to
\begin{corollary}\label{R:lambda1}
Suppose that $Q\in \mathbb P_{1,k}$. Then for $i+h\le \max\{k,5\}-4$ and  as $|\lambda|\to\infty$,
\[
|\partial_x^i\partial_y^h
\Psi(x,y,\lambda)-\partial_x^i\partial_y^h\left(1-\frac Q\lambda \right)|\le \frac { C}{|\lambda|^2}.
\]
\end{corollary}

\begin{remark}\label{R:limit}
In general, the scattering transformation is a  generalized Fourier transform. That is, it maps smooth potentials to decaying scattering data, and decaying potentials to smooth scattering data. As is known, the asymptotic expansion of eigenfunctions is related to the decayness of the scatterig data. However, in the case of Ward equation, even for the Schwartz potentials, the second order  asymptotic expansion of Theorem \ref{T:lambda2} seems difficult to be improved. To see it,  the second order coefficient of the asymptotic expansion $\Psi$, and an analogue of (\ref{E:psix1}) need to be introduced. That is
\begin{eqnarray*}
\Psi_2(x,y)&=&\int_{-\infty}^x \left(-Q_y+Q_{x}Q\right)(x',y)\,dx',\label{E:psi2-1}
\end{eqnarray*}
\begin{eqnarray*}
c(y)&=&\int_{-\infty}^\infty \left(-Q_y+Q_{x}Q\right)(x',y)\,dx',\label{E:psi2-2}\\
\Phi(x)&=&\int_{-\infty}^x \phi(x')\,dx',\qquad\int_{-\infty}^\infty \phi(x')dx'=1,\label{E:psi2-3}\\
f(x,y)&=&\Psi_2(x,y)-c(y)\Phi(x).\label{E:psi2-4}
\end{eqnarray*}
and
\begin{equation}
(\partial_y-\lambda\partial_x)\left(\Psi^{-1}(1-\frac Q\lambda+\frac {\Psi_2}{\lambda^2})\right)
=\frac 1{\lambda^2}\Psi^{-1}\left(\partial_y\Psi_2-Q_x\Psi_2\right)\label{E:psi-5}
\end{equation}
Where $\phi$ is a Schwartz function. Then $f(x,y),\,c(y)$ are Schwartz. It can be checked that $\Psi_2$ does not possess integrability in the $x$-variable. This causes troubles in estimating $|\Psi-(1-\frac Q\lambda+\frac{\Psi_2}{\lambda^2})|$ while inverting (\ref{E:psi-5}) to derive a higher order asymptotic expansion of $\Psi$. 
\end{remark}

%%%%%%%%%%%%%%%%%%%%%%%%%%%%%%%%%%%%%%%%%%%%%%%%%%%%%%%%%%%%%%%%%%%%%%%%%%%%%%%%%%%%%%%%%%%%%%%%%%%%%%%%%%%%%%%%%%%%%%%%%%%%%%
\section{Direct problem III: Eigenfunctions with non-small data} \label{S:ENS}
%%%%%%%%%%%%%%%%%%%%%%%%%%%%%%%%%%%%%%%%%%%%%%%%%%%%%%%%%%%%%%%%%%%%%%%%%%%%%%%%%%%%%%%%%%%%%%%%%%%%%%%%%%%%%%%%%%%%%%%%%%%%%%%

First we introduce
\begin{definition}\label{D:cauchy}
The Cauchy operator $\mathcal C$ and its limits 
  $\mathcal{C}_\pm$ are defined as:  
\[\begin{array}{ll}
\mathcal{C}f(\lambda)=\frac 1{2\pi i}\int^\infty_{-\infty}\frac {f(\zeta)}{\zeta-\lambda}d\zeta,&\lambda\in\mathbb{ C}\backslash \mathbb{R},\\
\mathcal{C}_\pm f(\lambda)=\lim_{\epsilon\to 0^+}\frac 1{2\pi i}\int^\infty_{-\infty}\frac {f(\zeta)}{\zeta-\left(\lambda\pm i\epsilon\right)}d\zeta,&\lambda\in \mathbb{R}.
\end{array}
\]
\end{definition}

It is well-known that  $\mathcal{C}_\pm$ are bounded operators on $L_p(\mathbb{R})$ for $1<p<\infty$, and $\mathcal{C}_\pm f(\lambda)=\lim_{\tilde\lambda\to \lambda}\mathcal{C}f(\tilde\lambda)$, $\lambda\in \mathbb R$,   $\tilde\lambda\in \mathbb{ C}^\pm$ \cite{SW}.

\begin{definition}\label{D:SMrph}
Suppose $v(\lambda)$ is defined on $\mathbb{R}$. A function $\Psi(\lambda)$ is called a solution of the Riemann-Hilbert problem $(\lambda\in\mathbb{R}, v)$ if
\begin{eqnarray*}\Psi(\lambda)=&&1+\frac 1{2\pi i}\int_{\mathbb R}\frac{\Psi_-(t)\left(v(t)-1\right)}{t-\lambda}dt\\
=&&1+\mathcal C\left(\Psi_-(v-1)\right).
\end{eqnarray*}
Where $\Psi_\pm(\lambda)=\lim_{\tilde\lambda\to \lambda}\Psi(\tilde\lambda)$, $\lambda\in \mathbb R$,   $\tilde\lambda\in \mathbb{ C}^\pm$. Moreover, the function $v(\lambda)$ is called the data of the Riemann-Hilbert problem $(\lambda\in\mathbb{R}, v)$. 
\end{definition}

Suppose the data $v(\lambda)$, $\lambda\in \mathbb R$ satisfies $\partial_\lambda^i\left(\Psi-1\right)\in L_2(\mathbb R, d\lambda)$, for $i=0,1,2$. It can be seen that $\Psi$ is a solution of the Riemann-Hilbert problem $(\lambda\in\mathbb{R}, v)$ if and only if 
\[\begin{array}{ll}
\textit{$\partial_{\bar \lambda}\Psi=0$}, &\textit{$\lambda\in \mathbb{C}^\pm$}, \\
\textit{$\Psi_+=\Psi_- v$}, &\textit{$\lambda\in\mathbb{R}$},\\
\textit{$\Psi\to 1$,}&\textit{as $|\lambda|\to\infty$.}
\end{array}
\]

\begin{lemma}\label{L:SMrm}  
Suppose the data $v(\lambda)$, $\lambda\in \mathbb{R}$, satisfies:
\begin{eqnarray*}
&&v-1\in L_2(d\lambda),  \\%\label{E:S1}\\
&&|v-1|_{L_\infty(d\lambda)}\Vert \mathcal{C}_\pm\Vert_{2}< 1.%\label{E:S3}
\end{eqnarray*}
Then the Riemann-Hilbert problem $(\lambda\in\mathbb{R}, v)$ has a unique solution $\Psi$ such that $\Psi-1\in  L_\infty(d\lambda)\cap L_2(d\lambda)$. Moreover, if $H^k=\{f|\partial_\lambda^j f\in L_2(d\lambda),\,0\le j\le k\}$ and
\[
| v-1|_{H^k(d\lambda)}<<1,
\] then
\[ | \Psi_\pm-1 |_{ H^k(d\lambda) }\le C| v-1|_{H^k(d\lambda)}%\label{E:SMrm}
\]for some constant $C$.
\end{lemma}
\begin{proof} The proof can be drived by an adaptation of the proof of Theorem 8.9 and 9.20 in \cite{BC84}. \end{proof}

\begin{lemma}\label{L:scalarrh}
Suppose the data $v(\lambda)$, $\lambda\in\mathbb R$, is a scalar function satisfying:
\begin{itemize}
	\item $v(\lambda)\ne 0$, $\forall \lambda$;
	\item $\int_{-\infty}^\infty d\,\mathrm {arg }\,v(\lambda)=0$;
	\item $v -1$, $\partial_\lambda v \in L_2(d\lambda)$.
\end{itemize}
Then the Riemann-Hilbert problem $(\lambda\in\mathbb{R}, v)$ has a unique solution $\Psi$.  Moreover, if
\[
v-1 \in H^k(d\lambda)
\] then
\[
 | \Psi_\pm-1 |_{ H^k(d\lambda) }\le C| v-1|_{H^k(d\lambda)}.%\label{E:scalarrm}
\]
Where $H^k(d\lambda)=\{f|\partial_\lambda^if\in L_2(d\lambda),\,0\le i\le k\}$, and $C$ is a constant depending on $|v|_{L_\infty}, |1/v|_{L_\infty}$.
\end{lemma}
\begin{proof} Note that by the Sobolev's theorem,  $v -1\in C_0$ by condition $v -1$, $\partial_\lambda v \in L_2(d\lambda)$. Here $C_0$ denotes continuous functions with limit $0$ at $\infty$. Hence the proof can be found in Appendix of \cite{BC84}. \end{proof}

\begin{lemma}\label{L:psil2}
Suppose $Q\in \mathbb P_{\infty,2,0}\cap \mathbb P_1$. Then the eigenfunction obtained in Theorem \ref{T:SMexistence} satisfies:
\begin{enumerate}
	\item $\partial_x^i\left(\Psi(\cdot,y,\lambda)-1\right)$, $i=0,1,2$, are uniformly bounded in $L_2(dx)$;
	%\item $(\Psi(x,0,\lambda)-1)\chi_{|x|>N}$, $(\partial_x\Psi(x,0,\lambda))\chi_{|x|>N}\to 0$ uniformly in $L_2(dx)$ as $N\to\infty$.%;
  \item $\Psi(\cdot,y,\lambda)-1,\,\partial_x\Psi(\cdot,y,\lambda)\to 0$ uniformly in $L_2(dx)$ as $\lambda\to\infty$.
\end{enumerate}
\end{lemma}

\begin{proof} By noting that the Fourier transform is an isometry on the $L_2$ spaces, to prove (1), it suffices to show that $\xi^i\widehat W$, $i=0,1,2$, are uniformly bounded in $ L_2(d\xi)$. We will only treat the case of $\lambda\in \mathbb{C}^+$ and $\xi\ge 0$ for simplicity. Other cases can be handled similarly.
Note
\[
%|Kf|_{L^\infty }&&\le \int_{-\infty}^y\,|\widehat {Q_x}(\xi,y')|_{L^1(\xi)} | f(\xi,y')|_{L^\infty(\xi)}dy'\\
%&&\le|\widehat {Q_x}(\xi,y)|_{L^1(\xi,y)}| f|_{L^\infty}\\
|\mathcal K_\lambda f(\xi,y,\lambda)|_{L_2(d\xi )}%\le &&\int_{-\infty}^y\,|\widehat {\partial_xQ}(\xi,y')|_{L_1(d\xi)}|  f(\xi,y',\lambda)|_{L_2(d\xi )}\,dy'\\
\le |\widehat {\partial_xQ}(\xi,y)|_{L_1(d\xi dy)}\,\sup_{y,\lambda}| f|_{L_2(d\xi )}.
\]
Denote $\widehat {\mathbb X}_2=\left\{f(\xi,y,\lambda):\mathbb R\times\mathbb R\times \mathbb C\to M_n(\mathbb C)|\sup_{y,\lambda}|f(\xi,y,\lambda)|_{L_2(d\xi)}<\infty\right\}$. So 
\[
\mathcal K_\lambda :\widehat {\mathbb X}\cap \widehat {\mathbb X}_2\to \widehat {\mathbb X}\cap \widehat {\mathbb X}_2,\qquad
\Vert \mathcal K_\lambda \Vert\le  |\widehat {\partial_xQ}(\xi,y)|_{L_1(d\xi dy)}.
\]
By the assumption $Q\in \mathbb P_{\infty,2,0}$, we have 
$\int_{-\infty}^y e^{i\lambda\xi(y-y')}\widehat {\partial_xQ}(\xi,y')dy'\in \widehat {\mathbb X}\cap \widehat {\mathbb X}_2$. 
Therefore the solution $\widehat W$ of (\ref{E:WK}) is in $\widehat {\mathbb X}\cap \widehat {\mathbb X}_2$. 
Moreover, one can derive 
\begin{eqnarray*}
\xi\widehat W
=%&&\int_{-\infty}^y  e^{i\lambda\xi(y-y')}(\xi -\xi'+\xi')\widehat{Q_{x}}\ast \widehat Wdy'+ \int_{-\infty}^y  e^{i\lambda\xi(y-y')}\xi \widehat{Q_{x}}dy'\\
&&\int_{-\infty}^y  e^{i\lambda\xi(y-y')}\left(\xi\widehat{Q_{x}}\right)\ast \widehat Wdy'+\int_{-\infty}^y  e^{i\lambda\xi(y-y')}\widehat{Q_{x}}\ast \left(\xi\widehat W\right)dy'\\
&&+ \int_{-\infty}^y  e^{i\lambda\xi(y-y')}\xi \widehat{Q_{x}}dy'
\end{eqnarray*}
from (\ref{E:W1}). As a result, we have $\xi\widehat W%=(1-K)^{-1}(\int_{-\infty}^y  e^{i\lambda\xi(y-y')}(\xi \widehat{Q_{x'}})\ast \widehat Wdy'+ \int_{-\infty}^y  e^{i\lambda\xi(y-y')}\xi \widehat{Q_{x'}}dy')
\in\widehat {\mathbb X}\cap \widehat {\mathbb X}_2$, if $Q\in \mathbb P_{\infty,2,0}\cap \mathbb P_1$.  
The same argument can prove $\xi^2\widehat W,\,\xi^3\widehat W\in\widehat {\mathbb X}\cap \widehat {\mathbb X}_2$, if $Q\in \mathbb P_{\infty,2,0}\cap \mathbb P_1$. Hence (1) is justified.

To prove (2), by the definition of $\widehat{\mathbb X}$ and result of (1), the function $\widehat{W}(\xi,y,\lambda)$ can be approximated uniformly by $g$ where 
\[
|\widehat{W}-g|_{L_2(d\xi)\cap L_1(d\xi)}<\epsilon,
\]
and $g$ is a linear combination of step functions in $\xi$ with uniformly bounded coefficients in $y$, $\lambda$. Hence 
\[
\left(\int_{-\infty}^\infty e^{i\xi x}g(\xi,y,\lambda)d\xi\right)\chi_{|x|>N}\to 0 \textit{  uniformly in $L_2(dx)$ as $N\to\infty$.}%\label{E:2}
\]
Where $\chi_{|x|>N}$ is the characteristic function of the set $\{|x|>N\}$. The above two inequalities imply that $(\Psi(x,y,\lambda)-1)\chi_{|x|>N}\to 0$ uniformly in $L_2(dx)$ as $N\to\infty$.  We can prove the case of $(\partial_x\Psi(x,y,\lambda))\chi_{|x|>N}$ by the similar method.
Combining with %the result (1), 
Theorem \ref{T:lambda2} and the Lebesque Convergence Theorem, 
one can prove (2).
\end{proof}

\begin{lemma}\label{L:xy-uniform}
Let $x+\lambda y=z$, $\partial_{\bar z}=\frac 12(\partial_x+i\partial_y)$, and $f_{\pm,z}(x,\lambda)=\lim_{|y|\to 0^\pm}f(x,y,\lambda)$. 
If $f(x,y,\lambda)$ is the solution of the Riemann-Hilbert problem $( x\in\mathbb{R}, F( x,\lambda))$ and 
\begin{gather*}
F(\cdot,\lambda)-1,\,\,\partial_xF(\cdot,\lambda),\,\,f_{\pm,z}(\cdot,\lambda)-1,\,\,\partial_xf_{\pm,z}(\cdot,\lambda)\to 0\\%\textit{ in $L_2(dx)$ as $|\lambda|\to\infty$,}\\
\textit{ in $L_2(dx)$ as $|\lambda|\to\infty$,}
\end{gather*}
then
\[
f(x,y,\cdot) \textit{ tends to $1$ uniformly as $|\lambda|\to\infty$.}\label{E:fF}
\]
\end{lemma}
\begin{proof} 

For $y=0$, the lemma follows from the Sobolev's Theorem,   Lemma \ref{L:SMrm} and the assumption on $f_{\pm,z}$, $F$.

For simplicity, we omit the words "` for $|\lambda|>>1$"' in the following proof.

Decompose $f(x,y,\lambda)$ into
\begin{eqnarray*}
&&f(x,y,\lambda)\\
=&&1+\frac 1{2\pi i}\int_{|t-x|<1}\frac{f_-\left(F(t,\lambda)-1\right)}{t-z}dt+\frac 1{2\pi i}\int_{|t-x|\ge 1}\frac{f_-\left(F(t,\lambda)-1\right)}{t-z}dt\\
=&&1+I(x,y,\lambda)+II(x,y,\lambda)
\end{eqnarray*}
Note  $f_-(F-1)(\cdot, \lambda)$ is uniformly Holder continuous by  the assumption on $F$, $f_\pm$ and the imbedding theorem of Morrey \cite{GT83}. 
Hence one has $I(x,y,\lambda)\to I_{\pm, z}(x,\lambda)$ uniformly as $y\to 0^\pm$ \cite {G66}. %Here $g_{\pm,z}(x,0)=\lim_{|y|\to 0^\pm}g(x,y,\lambda)$. 
The uniform convergence of $II(x,y,\lambda)\to II_{\pm, z}(x,\lambda)$ as $y\to 0^\pm$ can be justified by the Holder inequality. Moreover, one can check that this convergence is independent of $x$. As a result,  $f(x,y,\lambda)\to f_{\pm, z}(x,\lambda)$ uniformly as $y\to 0^\pm$. 

Since the lemma holds  on the $x$-axis. The uniform convergence provided above implies that: for any $\epsilon>0$, one can find $N_{\epsilon_1}$, $\delta_\epsilon$ such that $|f(x,y,\lambda)-1|<\epsilon$ for $\forall |\lambda|\ge N_{\epsilon_1}$, $\forall |y|\le \delta_\epsilon$. 
Besides,  by the Holder inequality, we can find $N_{\epsilon_2}$ such that $|f(x,y,\lambda)-1|<\epsilon$ for $\forall |\lambda|>N_{\epsilon_2}$, $|y|\ge \delta_\epsilon$.
Hence for any $\epsilon>0$,  we obtain 
\[|f(x,y,\lambda)-1|<\epsilon, \qquad\textit{$\forall |\lambda|>\max\{N_{\epsilon_1},N_{\epsilon_2}\}$.}\]
\end{proof}

\textbf{We can start to prove Theorem \ref{T:LNexistence}.}

\noindent\textit{Proof.} We will prove Theorem \ref{T:LNexistence} by induction on the norm of $|\widehat {\partial_xQ}(\xi,y)|_{L_1(d\xi dy)}$. %To prove the the theorem, we need to establish the following auxiliary conditions in each induction step.

\underline{\emph{Step 1: (The case of $n=0$)}}  

If $|\widehat {\partial_xQ}(\xi,y)|_{L_1(d\xi dy)}<(\frac 32)^0$, the existence and (\ref{E:bdry}) are proved by Theorem \ref{T:SMexistence}. The conditions (\ref{E:bdry''}), (\ref{E:bdry'}) and (\ref{E:bdry'''}) are shown by Theorem \ref{T:lambda2} and Lemma \ref{L:psil2}.
The holomorphic property comes from (\ref{E:eigen+}).

\underline{\emph{Step 2: (Transforming to a Riemann-Hilbert problem)}} 
  
Suppose Theorem \ref{T:LNexistence} holds for $|\widehat {\partial_xQ}(\xi,y)|_{L_1(d\xi dy)}<(\frac 32)^n$. 
Note the eigenfunction corresponding to a $y$-translate of $Q$ is the $y$-translate  of the eigenfunction. Thus after translation we may have
\[
\int_R\int_{-\infty}^0|\widehat {\partial_xQ}(\xi,y)|dyd\xi=\int_R\int^\infty _0|\widehat {\partial_xQ}(\xi,y)|dyd\xi <(\frac 32)^{n+1}\frac 12<(\frac 32)^n,
\]
for a potential ${\partial_xQ}(x,y)$ with $|\widehat {\partial_xQ}(\xi,y)|_{L^1(d\xi dy)}<(\frac 32)^{n+1}$. Let $\chi^\pm=\chi^\pm(y)\le 1$ be smooth real-valued functions such that
\[\begin{array}{lll}
\chi^-=\begin{cases}1,&\textsl{for $y\le 0$,}\\
0,&\textsl{for $y\ge 1$,}\end{cases}
&  {\partial_xQ}^-={\partial_xQ}(x,y)\chi^-(y),& |\widehat{{\partial_xQ}^-}|_{L_1(d\xi dy)}<(\frac 32)^n, \\
\chi^+=\begin{cases}1,&\textsl{for $y\ge 0$,}\\
0,&\textsl{for $y\le -1$,}\end{cases}
&{\partial_xQ}^+={\partial_xQ}(x,y)\chi^+(y),& |\widehat{{\partial_xQ}^+}|_{L_1(d\xi dy)}<(\frac 32)^n.
\end{array}\]
So $Q^\pm\in \mathbb P_{\infty,4,0}$ and $|\widehat {{\partial_xQ}^\pm}(\xi,y)|_{L^1(d\xi dy)}<(\frac 32)^{n}$. By the induction hypothesis there exist bounded  sets  $Z^\pm$ such that $Z^\pm\cap\left(\mathbb{C}\backslash \mathbb{R}\right)$ are discrete in $\mathbb{C}\backslash \mathbb{R}$ and for all $\lambda\in \mathbb{C}\backslash  Z^\pm$, $Q^\pm$ have eigenfunctions $\Psi^\pm$  which fulfill the statements of Theorem \ref{T:LNexistence}. Here we remark that the meaning of the notation $\Psi^+$ is different from that of $\Psi_+$. The former is a function defined in the half plane $y\ge 0$, the latter means $\lim_{\lambda_I\to 0^+}\Psi(x,y,\lambda)$.

Hence any eigenfunction $\Psi$ for $Q$, whenever it exists,  must be of the form
\begin{equation}\begin{array}{l}\Psi(x,y,\lambda)=\begin{cases}\Psi^-(x,y,\lambda)a^-(x+\lambda y, \lambda),& y\le 0,\\
\Psi^+(x,y,\lambda)a^+(x+\lambda y, \lambda),& y\ge 0.
\end{cases}
\end{array}\label{E:BP}
\end{equation}
Where for $y\in\overline{\mathbb R^\pm}$,
\begin{equation}
\begin{cases}
\textsl{$a^\pm(x+\lambda y, \lambda)$ is meromorphic in $\lambda\in\mathbb{C}\backslash \mathbb{R}$ with discrete poles,}\\
\textsl{$a^\pm(x,y,\lambda)$ satisfies (\ref{E:bdry}), (\ref{E:bdry''}),}\\
\textsl{$a^\pm_{\pm,z}(x,0,\lambda)$ satisfies (\ref{E:bdry'}), (\ref{E:bdry'''}).}
\end{cases}\label{E:BP2}
\end{equation}

Conversely, if we can find $a^\pm$ such that $a^\pm$ satisfies (\ref{E:BP2}) for $y\in\overline{\mathbb R^\pm}$ and $a^+(a^-)^{-1}(x,0,\lambda)=(\Psi^+)^{-1}\Psi^-(x,0,\lambda)$ (The invertibility of $a^\pm$, $\Psi^\pm$ is implied by Lemma \ref{L:SMdet}). Then we can define $\Psi(x,y,\lambda)$ by (\ref{E:BP}) and prove Theorem \ref{T:LNexistence} in case of $|\widehat {\partial_xQ}(\xi,y)|_{L_1(d\xi dy)}<(\frac 32)^{n+1}$. Therefore, we conclude this step by

\begin{lemma} \label{L:tRH} {\emph{(Transforming into a Riemann-Hilbert problem)}} 
 To prove Theorem \ref{T:LNexistence}, it is equivalent to solving: Find a bounded set $Z$, $f(\tilde x,\tilde y,\lambda)$, and $\tilde f(\tilde x,\tilde y,\lambda)$ such that  $Z^\pm\subset Z$ and 
\begin{itemize}
\item $Z\cap\left(\mathbb{C}\backslash \mathbb{R}\right)$ is discrete in $\mathbb{C}\backslash \mathbb{R}$.
	\item For $\lambda\in\mathbb{C}^+\backslash \left(\mathbb{R}\cup Z\right)$,  $f$ is the unique solution of the Riemann-Hilbert problem $(\tilde x\in\mathbb{R}, F(\tilde x,\lambda))$.
	\item For $\lambda\in\mathbb{C}^-\backslash \left(\mathbb{R}\cup Z\right)$,  $\tilde f$ is the unique solution of the Riemann-Hilbert problem $(\tilde x\in\mathbb{R}, F^{-1}(\tilde x,\lambda))$.
	\item $f,\,\tilde f$ are meromorphic in $\lambda\in\mathbb{C}\backslash \mathbb{R}$ with poles at the points of $Z\cap \left(\mathbb{C}\backslash\mathbb{R}\right)$.
	%\item $f,\,\tilde f$ satisfy (\ref{E:bdry}), (\ref{E:bdry''}).
	\item $f_{\pm,z},\,\tilde f_{\pm,z}$ satisfy (\ref{E:bdry'}), (\ref{E:bdry'''}).
\end{itemize}
Where 
\begin{equation}
\textit{$x+\lambda y=\tilde x+i\tilde y= z$, $\tilde x$, $\tilde y\in \mathbb{R}$} \label{E:CV1}
\end{equation}
and
\begin{equation}
F(\tilde x,\lambda)=\Psi^-(\tilde x,0,\lambda)^{-1}\Psi^+(\tilde x,0,\lambda).\label{E:Fequation}
\end{equation}
\end{lemma}

\begin{proof} Note that if $f,\,\tilde f$ exist for Lemma \ref{L:tRH}, then by Lemma \ref{L:xy-uniform} $f,\,\tilde f$ satisfy (\ref{E:bdry}), (\ref{E:bdry''}) as well. Therefore, the lemma can be proved by the change of variables 
(\ref{E:CV1}) (or (\ref{E:CV})) 
and setting
\begin{eqnarray*}
&& a^-(x+\lambda y,\lambda)=A^-(\tilde x,\tilde y,\lambda),\\
&& a^+(x+\lambda y,\lambda)=A^+(\tilde x,\tilde y,\lambda),
\end{eqnarray*} with $\tilde x$, $\tilde y\in\mathbb{R}$.
\begin{eqnarray*}
&&f(\tilde x,\tilde y,\lambda)=\begin{cases} (A^+)^{-1}(\tilde x,\tilde y,\lambda), &\textsl{for $\tilde y\ge 0$, $\lambda\in \mathbb{C}^+$,}\\( A^-)^{-1}(\tilde x,\tilde y,\lambda),& \textsl{for $\tilde y\le 0$,  $\lambda\in \mathbb{C}^+$}. \end{cases}\\
&&\tilde f(\tilde x,\tilde y,\lambda)=\begin{cases} (A^-)^{-1}(\tilde x,\tilde y,\lambda), &\textsl{for $\tilde y\ge 0$, $\lambda\in \mathbb{C}^-$,}\\ (A^+)^{-1}(\tilde x,\tilde y,\lambda),& \textsl{for $\tilde y\le 0$, $\lambda\in \mathbb{C}^-$}. \end{cases}
\end{eqnarray*}
in the above discussion.  
\end{proof}

\underline{\emph{Step 3: (Factorization: a diagonal problem, a Riemann-Hilbert problem with }}

\underline{\emph{small data and a rational function)}}   

For any square matrix $A$ we let $d^+_k(A)$ denote the upper $(k\times k)$ principal minors. Also let  $\beta_{ik}$, $i\le k$ be the minor of $A$ formed of the first $i$ rows, the first $i-1$ columns, and the $k$th column, and $\gamma_{ki}$ be the minor of $A$ formed of the first $i$ columns, the first $i-1$ rows, and the $k$th row. The following factorization theorem can be found in \cite{FF63}.

\begin{lemma}\label{L:tridecom}
 Suppose the principal minors $d^+_k(A)\ne 0$, for $1\le k\le n$. Then the  matrix $A$ can be represented as
\[
A=CSB,
\]where
\begin{gather*}
C=\left(\begin{array}{cccc}
1&{}&{}&{0}\\
\frac{\gamma_{21}}{\gamma_{11}}&1&{}&{}\\
\vdots  &\vdots &\ddots &{}\\
\frac{\gamma_{n1}}{\gamma_{11}}&\frac{\gamma_{n2}}{\gamma_{22}}&\cdots&1
\end{array}
\right),\qquad B=\left(\begin{array}{cccc}
1&\frac{\beta_{12}}{\beta_{11}}&{\cdots}&{\frac{\beta_{1n}}{\beta_{11}}}\\
{}&1&{\cdots }&\frac{\beta_{2n}}{\beta_{22}}\\
{} &{}  &\ddots &\vdots\\
{0}&{}&{}&1
\end{array}
\right),\\
S=\left(\begin{array}{cccc}
d^+_1(A)&{}&{}&{0}\\
{}&\frac{d^+_2(A)}{d^+_1(A)}&{}&{}\\
{} &{} &\ddots &{}\\
0&{}&{}&\frac{d^+_n(A)}{d^+_{n-1}(A)}
\end{array}
\right).
\end{gather*}
\end{lemma}

From now on, we only deal with  the case of $\lambda\in\overline{\mathbb{C}^+}$ for simplicity. The other case can be proved in an analogous argument.

\begin{lemma}\label{L:trifact}
For $\lambda\in \overline{\mathbb{C}^+}\backslash\left[Z^+\cup Z^-\right] $, we have a factorization
\[
F( \tilde x,\lambda)=\left(1+g_l\right)^{-1}\delta\left(1+g_u\right) , \]
where
\begin{eqnarray}
&&\textit{$\delta$ is diagonal and $g_u$ ($g_l$) is strictly upper (lower) triangular. }\label{E:dec1}\\
&&\textit{$\delta$, $g_u$, $g_l$ are  $\lambda$-meromorphic in $\mathbb{C}^+$ with poles at $\left[Z^+\cup Z^-\right]$. }\label{E:dec5}
\\%&&\textit{$\delta$ satisfies (\ref{E:bdry''}), $g_u(x,y,\cdot)$, $g_l(x,y,\cdot)$ tend to $0$ uniformly as $|\lambda|\to\infty$}\label{E:dec6}.
%&&\textit{$\delta-1$, $g_u$, $g_l\in C_0\cap L_2(d\tilde x)$, $\partial_{\tilde x}\delta$, $\partial_{\tilde x}g_u$, $\partial_{\tilde x} g_l\in L_2(d\tilde x)$ for} \label{E:dec2}\\
%&&\textit{$\lambda\in \mathbb{C}^+ \backslash \left[Z^+\cup Z^-\right]_+$.}\nonumber\\
&&\textit{$\partial_{ x}^i\left(\delta-1\right)$, $\partial_{ x}^ig_u$, $\partial_{ x}^ig_l$, $i=0,1,2$ are uniformly bounded in $L_2(d\tilde x)$ for } \label{E:dec2}\\
&&\textit{$\lambda\in\overline{ \mathbb{C}^+ }\backslash \cup_{\lambda_j\in \left[Z^+\cup Z^-\right]}D_\epsilon(\lambda_j)$. For any $z_j\in\mathbb C\backslash\mathbb R$, fixing $\epsilon_k$ for $\forall k\ne j$    }\nonumber\\
&&\textit{and letting $\epsilon_j\to 0$, these $L_2(dx)$-norms increase as $ {C_j}{\epsilon_j^{-h_j}}$ with }\nonumber\\
&&\textit{uniform constants $C_j$, $h_j>0$.}\nonumber\\
%&&\textit{$(\delta-1)\chi_{|x|>N}$, $(g_u)\chi_{|x|>N}$, $(g_l)\chi_{|x|>N}$, $(\partial_{\tilde x}\delta)\chi_{|x|>N}$, $(\partial_{\tilde x}g_u)\chi_{|x|>N}$,} \label{E:dec4}\\
%&&\textit{$(\partial_{\tilde x} g_l)\chi_{|x|>N}\to 0$ uniformly in $L_2(d\tilde x)$ as $N\to\infty$ for}\nonumber
%\\
%&&\textit{$\lambda\in \mathbb{C}^+ \backslash \cup_{\lambda_j\in \left[Z^+\cup Z^-\right]_+}D_\epsilon(\lambda_j)$. }\nonumber%\\
&&\textit{$\delta-1$, $g_u$, $g_l,\,\partial_x\delta,\,\partial_xg_u,\,\partial_xg_l\to 0$ \textit{ in $L_2(dx)$ as $\lambda\to\infty$.}\label{E:dec4}}%\\
%&&\textit{}\nonumber
\end{eqnarray}
%Where $F$ is defined by (\ref{E:Fequation}) and $\left[Z^+\cup Z^-\right]_+=\left(Z^+\cup Z^-\right)\cap \mathbb{C}^+$.
\end{lemma}

\begin{proof} By the same technique of the proof of Lemma \ref{L:SMdet}, one proves $\det\Psi^\pm=1$ for $\lambda\notin\mathbb R$. So $\det F\equiv 1$. As a result, if $d^+_i\left( F\right)(\tilde x_0, \lambda_0)=0$ for some $1\le i<n$, then $F$ must have a pole at $(\tilde x_0, \lambda_0)$. By $\det\Psi^\pm=1$ and (\ref{E:Fequation}), we obtain $\lambda_0\in \left[Z^+\cup Z^-\right]$. 

Therefore for $\lambda\in\overline{ \mathbb{C}^+ } \backslash \left[Z^+\cup Z^-\right]$, we obtain a factorization by Lemma \ref{L:tridecom}. The properties (\ref{E:dec1})-(\ref{E:dec4}) are implied by
\begin{eqnarray}
	&& \textit{$F(\tilde x,\lambda)$ is meromorphic in $\lambda\in \mathbb{C}^+$ with poles at $\left[Z^+\cup Z^-\right]$ at most};\label{E:ind1}\\
	%&& \textit{$F(\tilde x,\lambda)$ satisfies (\ref{E:bdry''});}\label{E:ind2}\\
	&& \textit{$F(\tilde x,\lambda)$ satisfies (\ref{E:bdry'}), (\ref{E:bdry'''})}\label{E:ind3}
\end{eqnarray}
which come from the induction hypothesis.  
\end{proof}

\begin{lemma}\label{L:dRH}
{\emph{(A diagonal Riemann-Hilbert problem)}}
For $\lambda\in \overline{\mathbb{C}^+} \backslash \left[Z^+\cup Z^-\right]$,  the Riemann-Hilbert problem $(\tilde x\in\mathbb{R}, \delta(\tilde x,\lambda))$ has a solution $\Delta(z,\lambda)$. Moreover, 
\begin{itemize}
		\item $\Delta$ is $\lambda$-meromorphic in $\mathbb C^+$ with poles at $\left[Z^+\cup Z^-\right]\cap \mathbb C^+$;
		%\item $\Delta$ satisfies (\ref{E:bdry}), (\ref{E:bdry''});
		\item $\Delta_{\pm,z}$ satisfies (\ref{E:bdry'}), (\ref{E:bdry'''}).
\end{itemize}
\end{lemma}

\begin{proof}
For $\lambda\in \overline{\mathbb{C}^+} \backslash \left[Z^+\cup Z^-\right]$, the matix $\delta$ is a diagonal matrix with nonvanishing entries. So the winding number of $\delta(\tilde x,\lambda)$ is well-defined by $N(\lambda)=-\frac 1{2\pi i}\int\frac d{dt}\arg \delta(t,\lambda) dt$. By (\ref{E:dec5}) and (\ref{E:dec2}), $N(\lambda)$ is a continuous integer-valued function for $x\in\overline{\mathbb{C}^+}\backslash \left[Z^+\cup Z^-\right]$.  Thus $N(\lambda)\equiv 0$ by (\ref{E:dec4}).

Combining with  (\ref{E:dec2}), and (\ref{E:dec4}), Lemma \ref{L:scalarrh} implies the existence of $\Delta$ which satisfies the Riemann-Hilbert problem $(\tilde x\in\mathbb{R}, \delta(\tilde x,\lambda))$,  (\ref{E:bdry'}), and (\ref{E:bdry'''}). %Moreover, (\ref{E:dec4}) and Lemma \ref{L:scalarrh} yield (\ref{E:bdry'''}) as well.

The meromorphic property of $\Psi(x,y,\cdot)$ is proved by (\ref{E:dec5}), and
\[\Delta(z,\lambda)=\exp\left\{\frac 1{2\pi i}\int_\mathbb{R}\frac {\log \delta(t, \lambda)}{t-z}\,dt\right\}.
\]  
\end{proof}

\begin{lemma}\label{L:approx}
For  $\lambda\in \overline{\mathbb{C}^+} \backslash \cup_{\lambda_j\in \left[Z^+\cup Z^-\right]} D_\epsilon(\lambda_j)$, there exists
\[R_\epsilon=\begin{cases} R_{\epsilon,u}(\tilde x,\tilde y,\lambda), &\textit{for $\tilde y\ge 0$, }%$\lambda\in\mathbb{C}^+\backslash \cup D_\epsilon(\lambda_j)$,}
\\ R_{\epsilon,l}(\tilde x,\tilde y,\lambda),& \textit{for $\tilde y\le 0$%, $\lambda\in\mathbb{C}^+\backslash \cup D_\epsilon(\lambda_j)$
}, \end{cases}
\]
such that 
\begin{eqnarray}
&&\textit{$|\Delta_{-,z}(1+(R_\epsilon)_{-,z})F(1+(R_\epsilon)_{+,z})^{-1}\Delta_{+,z}^{-1}(\tilde x,\lambda)-1|_{H^2(\mathbb{R}, d\tilde x)}<<1$ . }\label{E:ap1}\\
&&\textit{$|\Delta_{-,z}(1+(R_\epsilon)_{-,z})F(1+(R_\epsilon)_{+,z})^{-1}\Delta_{+,z}^{-1}(\tilde x,\lambda)-1|_{L_\infty}\Vert C_\pm\Vert<1$. }\label{E:ap2}
\\
%&&\textit{$\lambda\in \mathbb{C}^+ \backslash \left[Z^+\cup Z^-\right]_+$ and uniformly as $|\lambda|\to\infty$.}\nonumber\\
&&\textit{$(R_\epsilon)_u$ ($(R_\epsilon)_l$) is strictly upper (lower) triangular.}\label{E:ap6}\\
&&\textit{$R_\epsilon$ can be meromorphically extended in $\lambda\in\overline{\mathbb{C}^+}$ with poles at $Z^+\cup Z^-$.}\label{E:ap4}
\\
&&\textit{$R_\epsilon\in H^2(\mathbb R, d\tilde x)$ and is rational in $z\in\mathbb{C}^\pm$, with  finite simple poles } \label{E:ap5}\\ 
&&\textit{(independent of $\lambda$) and each corresponding residue is an off diagonal}\nonumber\\
&&\textit{matrix with only one nonzero entry. Moreover, the non-zero entry}\nonumber\\
&&\textit{tends to $0$ as $|\lambda|\to\infty$.}\nonumber
\end{eqnarray}
%Here $f_{\pm,z}(x)=\lim_{|y|\to 0^\pm}f$, for $x+\lambda y=\tilde x+i\tilde y=z$.
\end{lemma}

\begin{proof} By the condition (\ref{E:dec4}), there exists $\delta_\epsilon$ such that  $|g_u\chi_{|\lambda|>\delta_\epsilon}|_{H^2(d\tilde x)}<\epsilon$. Moreover, by (\ref{E:dec2}), for each $\lambda_0\in \overline{\mathbb{C}^+} \backslash \cup_{\lambda_j\in \left[Z^+\cup Z^-\right]} D_\epsilon(\lambda_j)$, $|\lambda_0|\le \delta_\epsilon$, there exists $N=N(\epsilon,\lambda_0)$  such that
 \[|g_u-p_{\epsilon,u}|_{H^2(d\tilde x)}<\epsilon\qquad\textsl{for $\lambda$ in a small neighborhood of $\lambda_0$. }\]
Where
\begin{eqnarray*}
&&p_{\epsilon,u}(z,\lambda)=\sum_{j=-N^2}^{N^2}g_u(\frac j{N},\lambda)P_\epsilon(z-\frac j{N}),\\
&&P_\epsilon(t)=\frac 1{t-i\epsilon}-\frac 1{t+i\epsilon}\textit{ is the Poisson kernel \cite{BC84} (Appendix A.2). }
\end{eqnarray*}
 One can check that  $p_{\epsilon,u}\in H^2(\mathbb R, d\tilde x)$ satisfies 
 (\ref{E:ap6}), (\ref{E:ap4}). Hence choosing a bigger $N$ or $\delta_\epsilon$, there exists a $z$-rational function, denoted as $\tilde p_{\epsilon,u} $,
  \[|g_u-\tilde p_{\epsilon,u}|_{H^2(d\tilde x)}<\epsilon\qquad\textsl{for $\forall\lambda\in \overline{\mathbb{C}^+} \backslash \cup_{\lambda_j\in \left[Z^+\cup Z^-\right]} D_\epsilon(\lambda_j)$},\]
 and $\tilde p_{\epsilon,u}$ satisfies 
 (\ref{E:ap6}), (\ref{E:ap4}).

Consequently, using (\ref{E:ind1}), (\ref{E:ind3}), Lemma \ref{L:trifact}, \ref{L:dRH}, and the off-diagonal form of $g_u$, one can find a $z$-rational function $R_u(z,\lambda)$ which is an approximation of $g_u$ on $z\in\mathbb{R}$ and satisfies (\ref{E:ap1})-(\ref{E:ap5}). 

The case of $g_l$ can be done in analogy.
\end{proof}
 
With Lemma \ref{L:approx}, one can find a solution to the small-data Riemann-Hilbert problem $(\tilde x\in\mathbb{R}, \Delta_{-,z}(1+( R_\epsilon)_{-,z})F(1+( R_\epsilon)_{+,z})^{-1}\Delta_{+,z}^{-1})$. However, it is difficult to analyze the meromorphic property of the solution in a neighborhood of points in $\left[Z^+\cup Z^-\right]$. Hence we need to improve Lemma \ref{L:approx}. First of all, 
let us denote $ \mathbb{C}^+ _\epsilon=\left\{\lambda\in \mathbb C^+\,|\,\lambda_I \ge\epsilon\right\}$,  and 
$ \left[Z^+\cup Z^-\right]_\epsilon^+=\left\{\lambda\in \left[Z^+\cup Z^-\right]\,|\,\lambda_I\ge\epsilon\right\}$ for simplicity.

\begin{lemma}\label{L:approx1}
For  $\lambda\in \mathbb{C}^+_ \epsilon$, there exist
\[\tilde R_\epsilon=\begin{cases}\tilde  R_{\epsilon,u}(\tilde x,\tilde y,\lambda), &\textit{for $\tilde y\ge 0$, }%$\lambda\in\mathbb{C}^+\backslash \cup D_\epsilon(\lambda_j)$,}
\\ \tilde R_{\epsilon,l}(\tilde x,\tilde y,\lambda),& \textit{for $\tilde y\le 0$%, $\lambda\in\mathbb{C}^+\backslash \cup D_\epsilon(\lambda_j)$
} \end{cases}
\]
%and  
%\begin{equation}
%\mathcal U_\epsilon(\lambda)=\prod_{\lambda_j\in \left[Z^+\cup Z^-\right]_\epsilon^+}\left(\frac {\lambda-\lambda_j}{\lambda+i}\right)^{k_j},\label{E:uepsilon}
%\end{equation} 
such that 
\begin{eqnarray}
&&\textit{$|\Delta_{-,z}(1+(\tilde R_\epsilon)_{-,z})F(1+(\tilde R_\epsilon)_{+,z})^{-1}\Delta_{+,z}^{-1}-1|_{H^2(\mathbb{R}, d\tilde x)}<<1$ . }\label{E:ap1'}\\
&&\textit{$|\Delta_{-,z}(1+(\tilde R_\epsilon)_{-,z})F(1+(\tilde R_\epsilon)_{+,z})^{-1}\Delta_{+,z}^{-1}-1|_{L_\infty}\Vert C_\pm\Vert<1$. }\label{E:ap2'}
\\
%&&\textit{$\lambda\in \mathbb{C}^+ \backslash \left[Z^+\cup Z^-\right]_+$ and uniformly as $|\lambda|\to\infty$.}\nonumber\\
&&\textit{$(\tilde R_\epsilon)_u$ ($(\tilde R_\epsilon)_l$) is strictly upper (lower) triangular.}\label{E:ap6'}
\\
&&\textit{$\tilde R_\epsilon$ is meromorphic in $\lambda\in\mathbb C^+_\epsilon$ with poles at $\left[Z^+\cup Z^-\right]_\epsilon^+$.}\label{E:ap4'}\\
&&\textit{$\tilde R_\epsilon\in H^2(\mathbb R, d\tilde x)$ and is rational in $z\in\mathbb{C}^\pm$, with  finite simple poles } \label{E:ap5'}\\ 
&&\textit{(independent of $\lambda$) and each corresponding residue is an off diagonal}\nonumber\\
&&\textit{matrix with only one nonzero entry. Moreover, the non-zero entry}\nonumber\\
&&\textit{tends to $0$ as $|\lambda|\to\infty$.}\nonumber
\end{eqnarray}
%Here $f_{\pm,z}(x)=\lim_{|y|\to 0^\pm}f$, for $x+\lambda y=\tilde x+i\tilde y=z$.
\end{lemma}
\begin{proof}  One can multiply $g_u$ ($g_l$ respectively) by product 
\[\mathcal P_{\epsilon,u}=\prod_{\lambda_j\in \left[Z^+\cup Z^-\right]_\epsilon^+}\left(\frac {\lambda-\lambda_j}{\lambda+i}\right)^{h_j}\]
so that $\mathcal G_{\epsilon,u}=\mathcal P_{\epsilon,u} g_u$ is holomorphic in $\lambda\in\mathbb C^+_\epsilon$. Then using  (\ref{E:dec2}) and the same argument as the proof of Lemma \ref{L:approx}, one can approximate $\mathcal G_{\epsilon,u}$ by a piecewise $z$-rational function $R'_{\epsilon,u}$. Let $\tilde R_{\epsilon,u}=\mathcal P_{\epsilon,u}^{-1}R'_{\epsilon,u}$.

Next, choose $k_j$ sufficiently large in $\mathcal U_\epsilon(\lambda)=\prod_{\lambda_j\in \left[Z^+\cup Z^-\right]_\epsilon^+}\left(\frac {\lambda-\lambda_j}{\lambda+i}\right)^{k_j}$ to make $\mathcal U_\epsilon\delta$, $\mathcal U_\epsilon\Delta$ holomorphic in $\lambda\in\mathbb C^+_\epsilon$. Hence the lemma  can be proved by  an adaptation of the proof of  Lemma \ref{L:approx}. (Note the factors $\mathcal U_{\epsilon}$, $\mathcal P_{\epsilon,u}$, $\mathcal P_{\epsilon,l}$ are cancelled out.)
\end{proof}

\begin{lemma}\label{L:RHsmdata}{\emph{ (A Riemann-Hilbert problem with small data)}} 
The Riemann-Hilbert problem $(\tilde x\in\mathbb{R}, \Delta_{-,z}(1+(\tilde R_{\epsilon,u})_{-,z})F(1+(\tilde R_{\epsilon,u})_{+,z})^{-1}\Delta_{+,z}^{-1})$ admits a solution $f_{\epsilon,s}(z,\lambda)$  for $\lambda\in \mathbb{C}^+ _\epsilon\backslash \left[Z^+\cup Z^-\right]_\epsilon^+$. 
Moreover, 
\begin{itemize}
		\item $f_{\epsilon,s}$  is meromorphic in $\lambda\in\mathbb C^+_\epsilon$ with poles at $\left[Z^+\cup Z^-\right]_\epsilon^+$.
    \item $\left(f_{\epsilon,s}\right)_{\pm,z}$  satisfies (\ref{E:bdry'}), (\ref{E:bdry'''}).
\end{itemize}
\end{lemma}

\begin{proof} By the assumption (\ref{E:ap1'}), (\ref{E:ap2'}), %for $\lambda\in \mathbb{C}^+ \backslash \cup_{\lambda_j\in \left[Z^+\cup Z^-\right]} D_\epsilon(\lambda_j)$, 
one can apply Lemma \ref{L:SMrm} to find $ f_{\epsilon,s}$ which satisfies (\ref{E:bdry'}) and the Riemann-Hilbert problem 
$(\tilde x\in\mathbb{R}, \Delta_{-,z}(1+(\tilde R_{\epsilon,u})_{-,z})F(1+(\tilde R_{\epsilon,u})_{+,z})^{-1}\Delta_{+,z}^{-1})$

Moreover, $ f_{\epsilon,s}$ satisfies (\ref{E:bdry'''})  by Lemma \ref{L:SMrm}, (\ref{E:ind3}), Lemma \ref{L:dRH},  and (\ref{E:ap5'}). Finally, $ f_{\epsilon,s}$ is is meromorphic in $\lambda\in\mathbb C^+_\epsilon$ with poles at $\left[Z^+\cup Z^-\right]_\epsilon^+$ by (\ref{E:ind1}), Lemma \ref{L:dRH}, and (\ref{E:ap4'}).  
\end{proof}

We conclude this step by a characterization of Lemma \ref{L:tRH}.

\begin{lemma}\label{L:FORHP}{\emph{ (Factorization of the Riemann-Hilbert problem)}}   
Suppose $f(z,\lambda)$ fulfills the statement in Lemma \ref{L:tRH}.
Then  there exist a unique function $r_\epsilon(z,\lambda)$ and a set $Z_\epsilon$,  such that %for $\lambda\in \mathbb{C}^+_\epsilon\backslash Z_\epsilon$,
\begin{equation}
\textit{$r_\epsilon(z,\lambda)=1+\sum_{k=1}^{N_\epsilon} (z-z_k)^{-1}c_{k,\epsilon}(\lambda)$ ,  }\label{E:rational1}
\end{equation}
for some integer $N_\epsilon$, $ Z_\epsilon \subset Z$, and for $\lambda\in\mathbb{C}^+_\epsilon\backslash Z$,
\begin{eqnarray}
&&\textit{$c_{k,\epsilon}$ is meromorphic in $\lambda\in\mathbb C^+_\epsilon$ with poles at $Z_\epsilon$,}\label{E:rational2}\\% and tends to $1$ uniformly 
%&&\textit{as $|\lambda|\to\infty$,}\nonumber\\
&&\textit{$c_{k,\epsilon}(\lambda)\to 0$ as $|\lambda|\to\infty$,}\label{E:rational3}\\
&&f=r_\epsilon f_{\epsilon,s}\Delta (1+\tilde R_\epsilon).\label{rational}
\end{eqnarray}

Conversely, suppose there are uniformly bounded  sets $Z_\epsilon$, %$\left[Z^+\cup Z^-\right]_\epsilon^+\subset Z_\epsilon$, %$\epsilon=1,\frac 12, \frac 13,\cdots$, 
and functions $\left\{r_\epsilon\right\}$  which are $\lambda$-meromorphic in $\mathbb C^+_\epsilon$ with poles at $Z_\epsilon$, satisfy (\ref{E:rational1})- (\ref{E:rational3}), and  
\begin{equation}
\textit{$r_\epsilon f_{\epsilon,s}\Delta (1+\tilde R_\epsilon)$ is holomorphic in $ z\in \mathbb{C}^\pm$}\label{E:rational4}
\end{equation}
for $\lambda\in\mathbb{C}^+_\epsilon\backslash \left(Z_\epsilon \cup \left[Z^+\cup Z^-\right]_\epsilon^+\right)$. Define $f_\epsilon=r_\epsilon f_{\epsilon,s}\Delta (1+\tilde R_\epsilon)$ for $\lambda\in\mathbb{C}^+_\epsilon$. Then we have
\begin{eqnarray}
&&f_\epsilon \textit{ is meromorphic in $\lambda\in\mathbb C^+_\epsilon$ with poles at $Z_\epsilon\cup\left[Z^+\cup Z^-\right]_\epsilon^+$,} \label{E:meroext}\\
&&f_{\epsilon_1}=f_{\epsilon_2}\qquad\textit{for $\lambda\in\mathbb{C}^+_{\epsilon_1}$ if $\epsilon_1>\epsilon_2$}.\label{E:liou}
\end{eqnarray}
Hence  $f=f_\epsilon$ is well-defined, %$Z_\epsilon$ are discrete in $\mathbb C^+$,  
and $f$ satisfies the statements in Lemma \ref{L:tRH} with $Z={\cup_\epsilon Z(f_\epsilon)}\cup \{\lambda_j\in\mathbb R|\limsup_{\epsilon\to 0}{|f_\epsilon(D_{2\epsilon}(\lambda_j)\cap C^+_\epsilon)|}=\infty\}$. Here $Z(f_\epsilon)$ denotes the poles of $f_\epsilon$.
\end{lemma}

\begin{proof} First of all, by Lemma \ref{L:SMdet},  $\det\,f_{\epsilon,s}(z,\lambda)=\det\,(1+\tilde R_\epsilon(z,\lambda))=\det\,\Delta (z,\lambda)=1$. So they are invertible at regular $\lambda$. Besides, $f(z,\lambda)$ and $f_{\epsilon,s}(z,\lambda)\Delta (z,\lambda)(1+\tilde R_\epsilon(z,\lambda)) $ are $z$-meromorphic, possess the same jump singularity across $z\in\mathbb{R}$, and tend to $1$ at infinity. Therefore 
\[f\left[f_{\epsilon,s}(z,\lambda)\Delta (z,\lambda)(1+\tilde R_\epsilon(z,\lambda))\right]^{-1}
\]
is $z$-rational and (\ref{E:rational1})-(\ref{E:rational3}) are satisfied by Lemma \ref{L:dRH}-\ref{L:RHsmdata} and the assumption on $f$. %Indeed the poles of $r_\epsilon$ are identical with those of $R_\epsilon$, the residue $c_k$ are determined by $f$, $f_{\epsilon,s}$, $\Delta$.

For the converse part, (\ref{E:meroext}) comes immediately from the definition of $f_\epsilon$ and the meromorphic  properties of $r_\epsilon$, $\Delta $, $\tilde R_\epsilon$, $ f_{\epsilon,s}$ implied by assumption and Lemma  \ref{L:dRH}-\ref{L:RHsmdata}.

Besides, by assumption, $f_{\epsilon_1}$, $f_{\epsilon_2}$ satisfy the same Riemann-Hilbert problem in Lemma \ref{L:tRH} for  $\lambda\in\mathbb{C}^+_{\epsilon_1}\backslash Z_{\epsilon_1}$. Thus (\ref{E:liou}) follows from the Liouville's theorem and the meromorphic properties. As a result, the well-defined property follows from (\ref{E:meroext}) and  (\ref{E:liou}).

%Moreover, by construction $r$ is $\lambda$-holomorphic  outside the discrete set $Z$ by the meromorphic properties of $F$, $\Delta$, $\delta$, $g_u$, $g_l$. So the singularities $Z$ are of pole type. 

The conditions (\ref{E:bdry'}), (\ref{E:bdry'''}) can be proved by Lemma  \ref{L:dRH}-\ref{L:RHsmdata},  and (\ref{E:rational1})- (\ref{E:rational3}), $f=f_\epsilon$ (i.e., (\ref{rational})), and $Z={\cup Z_\epsilon}\cup \{\lambda_j\in\mathbb R|\limsup_{\epsilon\to 0}{|f_\epsilon(D_{2\epsilon}(\lambda_j)\cap C^+_\epsilon)|}=\infty\}$.
\end{proof}

\underline{\emph{Step 4: (Solving the Riemann-Hilbert problem)}}   

\vskip.1in
We complete the proof of Theorem \ref{T:LNexistence} by finding a rational function $r_\epsilon$ in Lemma \ref{L:FORHP}.

\begin{lemma}\label{L:linearsystem}
{\emph{ (Existence of the rational function $r_\epsilon$)}}   
There exist a   function $r_\epsilon$ and a uniformly bounded set $Z_\epsilon$ such that  $r_\epsilon$ is $\lambda$-meromorphic in $\mathbb C^+_\epsilon$ with poles at the points of $Z_\epsilon$ and satisfies (\ref{E:rational1})-(\ref{E:rational3}), (\ref{E:rational4}) for $\lambda\in\mathbb{C}^+_\epsilon\backslash \left(Z_\epsilon\cup\left[Z^+\cup Z^-\right]_\epsilon^+\right)$.
\end{lemma}

\begin{proof} For simplicity, we drop $\epsilon$ in the notation $r_\epsilon$, $f_{\epsilon,s}$, $R_\epsilon, \cdots$ in the following proof.

{\emph{ (a) A linear system for $r(z,\lambda)$:}}  

Let $\left\{z_k=\tilde x_k+i\tilde y_k\right\}$, $k=1,\cdots,N$ be the simple poles of $R$ in $\mathbb{C}^\pm$ by (\ref{E:ap5}). Denote
\begin{eqnarray}
&&1+R(z,\lambda)=(z-z_j)^{-1}d_j+n_j+O(|z-z_j|),\label{E:1+R}\\
&&f_{s}\Delta (z,\lambda)=\alpha_j+\beta_j(z-z_j)+O(|z-z_j|^2).\label{E:fd}
\end{eqnarray}
at $z_j$. Thus
\[f_{s}\Delta (1+R)(z,\lambda)=(z-z_j)^{-1}\alpha_jd_j+(\beta_jd_j+\alpha_jn_j)+O(|z-z_j|).\]

Now let
\begin{equation}
r(z,\lambda)=1+\sum_{k=1}^N (z-z_k)^{-1}c_k.\label{E:form}
\end{equation}
Hence at $z_j$, 
\[r(z,\lambda)=(z-z_j)^{-1}c_j+b_j+O(|z-z_j|),
\]
where
\begin{equation}
 \begin{array}{l}
 b_j=1+\sum_{k\ne j}(z_j-z_k)^{-1}c_k.\\
  %\hat b_j=1+\sum_{k}(\hat z_j-z_k)^{-1}c_k+\sum_{k\ne j}(\hat z_j-\hat z_k)^{-1}\hat c_k.
  \end{array}\label{bj}
\end{equation}
We then try to find $c_j$, such that $r(z,\lambda)f_{s}(z,\lambda)\Delta(z, \lambda)(1+R(z,\lambda)) $  
is holomorphic at $z_j$. This yields the linear system for $c_j$:
\begin{eqnarray}
&&c_j\alpha_jd_j=0,\qquad\qquad\qquad\qquad1\le j\le N,\label{bcd1}\\
&&b_j\alpha_jd_j+c_j(\beta_jd_j+\alpha_jn_j)=0, \,1\le j\le N.\label{bcd2}
\end{eqnarray}

{\emph{ (b) Solving the linear system (\ref{bcd1}),(\ref{bcd2}):}}  

The properties (\ref{E:ap6}), (\ref{E:ap5}) imply that $n_j$ are invertible and  $(d_jn_j^{-1})^2=0$. Therefore, 
it can be justified that (\ref{bcd1}) are consequences of (\ref{bcd2}). %Note the off-diagonal form of $g_l$ ($g_u$) in Lemma \ref{L:trifact} is crucial here.

Inserting (\ref{bj}) into (\ref{bcd2}), we obtain a system of $Nn^2$ linear equations in $Nn^2$ unknowns (the entries of $c_k$) with coefficients in entries of $d_j(\lambda)$, $n_j(\lambda)$, $\alpha_j(\lambda)$, $\beta_j(\lambda)$. Observing that as $|\lambda|\to\infty$, 
\[ d_j\to 0,\,\,n_j\to 1,\,\, \alpha_j\to 1,\,\, \beta_j\to 0 \]
by Lemma \ref{L:dRH}-\ref{L:RHsmdata}. Therefore, (\ref{bcd2}) are solvable as $|\lambda|\to\infty$. Precisely, $c_k$ can be written in  rational forms of $ d_j$, $n_j$, $\alpha_j$, $\beta_j$ which are all holomorphic in $\lambda\in \mathbb{C}^+_\epsilon\backslash \left[Z^+\cup Z^-\right]$. Therefore, (\ref{bcd2}) are solvable for $\lambda \in \mathbb{C}^+\backslash  Z_\epsilon$ where $Z_\epsilon$ are uniformly bounded  sets. Consequently, (\ref{E:rational1}), (\ref{E:rational2}), (\ref{E:rational3}), and (\ref{E:rational4}) are fulfilled.  
%Finally, the property  is implied by %the $\lambda$-boundary behavior of $ d_j$, $n_j$, $\alpha_j$, $\beta_j$, $\hat d_j$, $\hat  n_j$, $ \hat \alpha_j$, $ \hat \beta_j$ and (\ref{E:form})
%. 
\end{proof}

By the same argument as  the proof of Theorem \ref{T:LNexistence}, we have
%\begin{corollary}\label
%Suppose that $Q\in \mathbb P_{\infty,3}$ and $\Psi(x,y,\lambda)$ is the associated eigenfunction. 
% We have
%\begin{gather*}
%\Psi(\cdot,0,\lambda)-1,\,\partial_x\Psi(\cdot,0,\lambda)\,\,\textit{are uniformly bounded in $L_2(dx)$}\\
%\textit{for $\forall\lambda\in \mathbb{C}\backslash \left(\mathbb{R} \cup_{\lambda_j\in Z} D_\epsilon(\lambda_j)\right)$.}
%\end{gather*}  
%In particular, if $\lambda_0$ is a removable singularty of $\Psi(x,y,\lambda)$, then 
%\begin{gather*}
%\Psi(\cdot,0,\lambda_0)-1,\,\partial_x\Psi(\cdot,0,\lambda_0)\,\,\textit{are uniformly bounded in $L_2(dx)$}\\
%\textit{in a neighborhood of $\lambda_0$.}
%\end{gather*}  
%\end{corollary}

\begin{corollary}\label{C:boundedness}
Suppose that $Q\in \mathbb P_{\infty,k,0}$, $k\ge 2$ and $\Psi(x,y,\lambda)$ is the associated eigenfunction. 
 Then
\begin{gather*}
\Psi-1\textit{ are uniformly bounded in }\mathbb {DH}^{k} 
\textit{ for $\lambda\in \mathbb{C}\backslash \left(\mathbb{R} \cup_{\lambda_j\in Z} D_\epsilon(\lambda_j)\right)$.}
\end{gather*}  
In particular, if $\lambda_0$ is a removable singularty of $\Psi(x,y,\lambda)$, then 
\begin{gather*}
\Psi-1\textit{ are uniformly bounded in }\mathbb {DH}^{k}
\textit{ in a neighborhood of $\lambda_0$.}
\end{gather*}  
\end{corollary}

By a similar argument as that in Lemma \ref{L:SMdet}  
 and \ref{L:SMreality} and using the uniqueness property in Theorem \ref{T:LNexistence}, we can derive the same algebraic characterization of the eigenfunctions:

\begin{lemma}\label{L:symmetrydp}
Suppose that $Q\in \mathbb P_{\infty,k,0}$, $k\ge 2$. Then the eigenfunction $\Psi$ satisfies 
\begin{eqnarray}
&&\det\Psi(x,y,\lambda)\equiv 1,\label{E:det2}\\
&&\Psi(x,y,\lambda)\Psi(x,y,\bar\lambda)^*=I.\label{E:reality2}
\end{eqnarray}
for $\lambda\in\mathbb{C}\backslash \mathbb{R}$.
\end{lemma}

%%%%%%%%%%%%%%%%%%%%%%%%%%%%%%%%%%%%%%%%%%%%%%%%%%%%%%%%%%%%%%%%%%%%%%%%%%%%%%%%%%%%%%%%%%%%%%%%%%%%%%%%%%%%%%%%%%%%%%%%%%%%%%
\section{Direct problem IV: Asymptotic analysis with non-small data} \label{S:CSD}
%%%%%%%%%%%%%%%%%%%%%%%%%%%%%%%%%%%%%%%%%%%%%%%%%%%%%%%%%%%%%%%%%%%%%%%%%%%%%%%%%%%%%%%%%%%%%%%%%%%%%%%%%%%%%%%%%%%%%%%%%%%%%%%
We define the continuous scattering data and study its algebraic and analytic characteristics in this section. We first show that the existence of continuous scattering data for $Q\in\mathbb P_1$ is automatic.

\begin{lemma}\label{L:pmexistence}
If $Q\in \mathbb P_1$, then the eigenfunction $\Psi(x,y,\cdot)$ obtained by Theorem \ref{T:SMexistence}  has limits $\Psi_\pm$ on $\mathbb{R}$.
\end{lemma} 

\begin{proof} Suppose $\left\{\lambda_k\right\}\subset \mathbb{C}^+$.  Write $\widehat W_k$ instead of $\widehat W(\xi,y,\lambda_k)$ and 
\[f_k=\begin{cases}
\int_{-\infty}^y e^{i\lambda_k\xi(y-y')}\widehat {\partial_xQ}(\xi,y')dy',\qquad \textsl{when $\xi\ge 0$} \\
-\int^{\infty}_y e^{i\lambda_k\xi(y-y')}\widehat {\partial_xQ}(\xi,y')dy',\qquad \textsl{when $\xi\le 0$.}  
\end{cases}
\]%To study the existence of the limits, note that for $\forall x$, $y$
%\begin{eqnarray*}
%&&W(x,y,\lambda_{k})-W(x,y,\lambda_h)\\
%=&&
%\frac 1{2\pi}\int_\mathbb{R}e^{i\xi x}\left(\widehat W(\xi,y,\lambda_k)-\widehat W(\xi,y,\lambda_h)\right)d\xi\\
%\le && \frac 1{2\pi}|\widehat W(\xi,y,\lambda_k)-\widehat W(\xi,y,\lambda_h)|_{L_1(d\xi)}.
%\end{eqnarray*}
%Thus it reduces to showing that $|\widehat W(\xi,y,\lambda_k)-\widehat W(\xi,y,\lambda_h)|_{L_1(d\xi)}\to 0$. The following proof only deals with the case of $\left\{\lambda_k\right\}\subset \mathbb{C}^+$. The other case can be proved similarly. 
% and $\xi\ge 0$. 
Then (\ref{E:W1}) and (\ref{E:WK}) imply
\begin{eqnarray}
&&\widehat W_k-\widehat W_h\nonumber\\
=&&\left(1-K_{\lambda_k}\right)^{-1}\left(K_{\lambda_k}-K_{\lambda_h}\right)\widehat W_h+\left(1-K_{\lambda_k}\right)^{-1}\left(f_k-f_h\right)\nonumber\\
=&&I_1+I_2.\label{E:fk}
\end{eqnarray}
 Now observing
\begin{eqnarray}
I_1=&&\left(1-K_{\lambda_k}\right)^{-1}\left(K_{\lambda_k}-K_{\lambda_h}\right)\widehat W_h\nonumber\\
=&&\sum_{i=0}^N K_{\lambda_k}^i\left(K_{\lambda_k}-K_{\lambda_h}\right)\widehat W_h+K_{\lambda_k}^{N+1}\sum_{i=0}^\infty K_{\lambda_k}^i\left(K_{\lambda_k}-K_{\lambda_h}\right)\widehat W_h\nonumber\\
=&&I_1'+I_1''.\label{E:kh1}
\end{eqnarray}
Note (\ref{E:K}) and $\sup_y|\widehat W_h|_{L_1(d\xi)}\le (1-|\widehat {\partial_xQ}(\xi,y)|_{L_1(d\xi dy)})^{-1}$ imply \[\sup_y|\sum_{i=0}^\infty K_{\lambda_k}^i\left(K_{\lambda_k}-K_{\lambda_h}\right)\widehat W_h|_{L_1(d\xi)}<C'\]
and
\begin{equation}
|I_1''|_{L_1(d\xi)}=\sup_y|K_{\lambda_k}^{N+1}\sum_{i=0}^\infty K_{\lambda_k}^i\left(K_{\lambda_k}-K_{\lambda_h}\right)\widehat W_h|_{L_1(d\xi)}\to 0,\,\,\textsl{ as } N\to \infty\label{E:term2}
\end{equation}
On the other hand, 
\begin{eqnarray*}
&&|\left(K_{\lambda_k}-K_{\lambda_h}\right)\widehat W_h|_{L_1(d\xi)}\\
%\le && \int_0^\infty\int_{-\infty}^y\,|e^{i\lambda_k\xi(y-y')}-e^{i\lambda_h\xi(y-y')}|\,|\left(\widehat {Q_x}\ast \widehat W_h\right)(\xi,y',\lambda)|\,dy'd\xi\\
%&&+\int^0_{-\infty}\int_y^{\infty}\,|e^{i\lambda_k\xi(y-y')}-e^{i\lambda_h\xi(y-y')}|\,|\left(\widehat {Q_x}\ast \widehat W_h\right)(\xi,y',\lambda)|\,dy'd\xi\\
\le && \int_{-\infty}^y\,|(e^{i\lambda_k\xi(y-y')}-e^{i\lambda_h\xi(y-y')}|\,|\widehat {\partial_xQ}|_{L_1(d\xi)}  |\widehat W_h|_{L_1(d\xi)}\,dy'\\
&&+\int_y^{\infty}\,|(e^{i\lambda_k\xi(y-y')}-e^{i\lambda_h\xi(y-y')}|\,|\widehat {\partial_xQ}|_{L_1(d\xi)}|\widehat W_h|_{L_1(d\xi)}\,dy'\\
\to &&0,\qquad\textsl{ as }k,\,h\to\infty.
\end{eqnarray*}
by the Lebesque Convergence Theorem and $Q\in \mathbb P_1$. 
So 
\begin{equation}
|I_1'|_{L_1(d\xi)}=|\sum_{i=0}^N K_{\lambda_k}^i\left(K_{\lambda_k}-K_{\lambda_h}\right)\widehat W_h|_{L_1(d\xi)}\to 0,\qquad\textsl{ as $k,\,h\to\infty$}.\label{E:t1}
\end{equation}
Hence 
$|I_1|_{L_1(d\xi)}\to 0$ as $k,\,h\to\infty$ by (\ref{E:kh1})-(\ref{E:t1}). 
A similar argument will induce  $|I_2|_{L_1(d\xi)}=|\left(1-K_{\lambda_k}\right)^{-1}\left(f_k-f_h\right)|_{L_1(d\xi)}\to 0$ as well.  Therefore, we have $|\widehat W_k-\widehat W_h|_{L_1(d\xi)}\to 0$ as $k,\,h\to\infty$ by (\ref{E:fk}). Taking the Fourier transform, we prove the lemma when $\lambda \in\mathbb{C}^+$. 

The case of $\lambda \in\mathbb{C}^-$ can be proved by analogy.
\end{proof}

\begin{lemma}\label{L:qxx}
Suppose that $Q\in \mathbb P_1$ and
\begin{equation}
|\xi^2\widehat {Q}|_{L_1(d\xi dy)}<\infty.\label{E:qxx}
\end{equation}
Then $\Psi_+$ and $\Psi_-$ are continuously differentiable with respect to $x$ and $y$. 
\end{lemma} 

\begin{proof} If $\lambda_k\to\lambda_\pm$ and $I_1$, $I_2$ are closed intervals on $\mathbb R$, 
\begin{itemize}
		\item $\partial_x\Psi(x,y,\lambda_k)$, and  $\partial_y\Psi(x,y,\lambda_k)$ are Cauchy for each $(x,y)\in I_1\times I_2$;
\item $\partial_x\Psi(x,y,\lambda_k)$, and  $\partial_y\Psi(x,y,\lambda_k)$ are uniformly bounded on $ I_1\times I_2$,
\end{itemize}
then $\Psi_\pm$ is differentiable and $\partial_x\Psi_\pm=\left(\partial_x\Psi\right)_\pm$, and  $\partial_y\Psi_\pm=\left(\partial_y\Psi\right)_\pm$ by the Lebesque Convergence theorem. Therefore, the continuous differentiability will be implied by proving the uniform Cauchy property of $\partial_x\Psi(x,y,\lambda_k)$, and  $\partial_y\Psi(x,y,\lambda_k)$ with respect to $x$, $y$ in compact subsets.  

Lemma \ref{L:pmexistence} and (\ref{E:Lax1}) imply that the uniform convergence of $\partial_y\Psi(x,y,\lambda_k)$ comes from that of $\partial_x\Psi(x,y,\lambda_k)$. 
So it is sufficient to show 
\[
|\xi\widehat W(\xi,y,\lambda_k)-\xi\widehat W(\xi,y,\lambda_h)|_{L_1(d\xi)}\to 0.
\]
By replacing $\widehat {\partial_x Q}(\xi,y')$ with $\xi\widehat {\partial_x Q}(\xi,y')$ in the representation of  $f_k$ in (\ref{E:fk}), it can be shown by adopting a similar argument as that in the proof of Lemma \ref{L:pmexistence}. 
\end{proof}

\begin{lemma}\label{L:lambda}
For $Q\in\mathbb P_1$ and $Q$ satisfies (\ref{E:qxx}), the eigenfunction $\Psi(x,y,\cdot)$ is holomorphic in $\mathbb{C}^\pm$ and has limits $\Psi_\pm$ on $\mathbb{R}$. Moreover, there exists a continuously differentiable function $v(x+\lambda y, \lambda)$ such that 
\[
\Psi_{+}(x,y,\lambda)=\Psi_-(x,y,\lambda)v(x+\lambda y, \lambda),\qquad \mathcal L_\lambda v=0, \qquad\lambda\in R,\label{E:CSD}
\]
where $\mathcal L_\lambda=\partial_y-\lambda\partial_x$.
\end{lemma}
\begin{proof} 
The holomorphicy has been proved in Theorem \ref{T:LNexistence}. By assumption, Lemma \ref{L:pmexistence} and \ref{L:SMdet}, $\Psi_\pm$ is invertible. Hence Lemma \ref{L:qxx} implies 
\begin{eqnarray*}
&&(\partial_y-\lambda\partial_x)\left\{\Psi_-^{-1}\Psi_+\right\}\\
=&&\left\{(\partial_y-\lambda\partial_x)\Psi_-^{-1}\right\}\Psi_++\Psi_-^{-1}(\partial_y-\lambda\partial_x)\Psi_+\\
=&&-\Psi_-^{-1}\left({\partial_x Q}\right)\Psi_++\Psi_-^{-1}\left({\partial_x Q}\right)\Psi_+\\
=&& 0
\end{eqnarray*} 
\end{proof}

%the eigenfunction $\Psi(x,y,\cdot)$ has limits $\Psi_\pm$ on $\mathbb{R}$
We denote  $Z=Z(\Psi)=\phi$ if there are no poles of $\Psi(x,y,\lambda)$.

\begin{lemma}\label{L:qxx'}
For $Q\in {\mathbb P}_{\infty,k,0}$, $k\ge 2$, if $Z=\phi$, then there exists a continuously differentiable function $v(x+\lambda y, \lambda)$ such that 
\[
\Psi_{+}(x,y,\lambda)=\Psi_-(x,y,\lambda)v(x+\lambda y, \lambda),\qquad \mathcal L_\lambda v=0, \qquad \lambda\in R.
\]
\end{lemma}
\begin{proof} Since $Z=\phi$, the eigenfunction $\Psi(x,y,\cdot)$ has limits $\Psi_\pm$ by Corollary \ref{C:boundedness} and the Sobolev's theorem. Moreover, we have the uniform convergence of $\partial_x\Psi(x,y,\lambda_k)$, $\lambda_k\to\lambda_0$. Hence the lemma can be proved by using the same argument as  the proof of Lemma \ref{L:lambda}.
\end{proof}

Since we are going to solve the inverse problem by the Riemann-Hilbert problem $(\lambda\in\mathbb R, v)$. By the scheme of Section \ref{S:ENS}, we need to investigate $L_2(\mathbb R, d\lambda)$ condition on $v$ and $\partial_\lambda v$. Hence the $\lambda$-asymptote of $v$ and $\partial_\lambda v$ will be investigated in the remaining part of this section.

We extend Theorem \ref{T:lambda2}, and Corollary \ref{R:lambda1} by:
\begin{lemma} \label{T:CSDana}
If $Q\in {\mathbb P}_{\infty,k,0}$, $k\ge 5$ and $Z=\phi$, then for $i+j\le k-4$, 
\begin{eqnarray*}
%&&|\Psi_\pm-\left(1-\frac Q\lambda \right)|\le \frac { C}{|\lambda|^2},\label{E:ana0}\\
&&|\partial_x^i\partial_y^j\left(\Psi_\pm-(1-\frac {\partial_x Q}\lambda)\right) |\le  \frac { C}{|\lambda|^2},\,\label{E:ana5}
\end{eqnarray*}
as $|\lambda|\to\infty$. Where $C$ is a constant depending on $Q$.
\end{lemma}

\begin{proof}  We follow the scheme in Section \ref{S:DPSMasymp} to prove this lemma. Note that all of the arguments there can be repeated except the proof of Lemma \ref{L:asym-1}. Where 
 the small data condition has been used to assure the uniform boundedness of $\partial_x^N\Psi$, $0\le N\le k-1$. Hence to prove this lemma, one needs only to show  
\begin{equation}
\textit{The uniform boundedness of $\partial_x^N\Psi_\pm$, $0\le N\le k-1$ as $|\lambda|\to\infty$.
}\label{E:basic1}
\end{equation}
However, since $Q\in {\mathbb P}_{\infty,k,0}$, $k\ge 5$, $\Psi_\pm$ exists, the property (\ref{E:basic1})  can be justified by  Corollary \ref{C:boundedness} and the Sobolev's theorem. 
\end{proof}

We improve the boundary properties (\ref{E:bdry}), (\ref{E:bdry''}) of Theorem \ref{T:LNexistence} by:
\begin{lemma}\label{L:boundedness} 
If $Q\in {\mathbb P}_{\infty,k,0}$, $k\ge 5$, and $Z=\phi$, then for $i+j\le k-4$,
\[
\partial_x^i\partial_y^j\left(\Psi_\pm-1\right) \to 0\,\textit{ uniformly in $L_\infty$ as $|x|$ or $|y|\to\infty$}.
\]
\end{lemma}
\begin{proof} 
By the results of Lemma \ref{T:CSDana}, it is sufficient to prove this lemma for $|\lambda|<c$ where $c$ is any fixed constant. However, for $|\lambda|<c$, $i+j\le k-4$,
\[
\partial_x^i\partial_y^j\left(\Psi_\pm(x,0,\lambda)-1\right)\to 0\,\textit{ uniformly in $L_\infty$ as $|x|\to\infty$}
\]follow from  (\ref{E:Lax1}), Corollary \ref{C:boundedness}, and the Sobolev's theorem.  For $y\ne 0$, one can follow the argument of Lemma \ref{L:xy-uniform} to show the uniform convergence of $\partial_x^i\partial_y^j\Psi\to \partial_x^i\partial_y^j\Psi_{\pm,z}$. Then  the lemma is proved by the uniform convergence and applying Holder inequality to 
\[
\Psi(x,y,\lambda)=1+\frac 1{2\pi i}\int^\infty_{-\infty}\frac {\Psi_+(t,0,\lambda)-\Psi_-(t,0,\lambda)}{t-(x+\lambda y )}dt.\]
\end{proof}

\begin{lemma}\label{L:lambdal2'}
For $Q\in {\mathbb P}_{\infty,k,1}\cap \mathbb P_1$, $k\ge 7$, we have 
\[
|\partial_\lambda\Psi_\pm|,\,\,|\partial_\lambda\partial_x\Psi_\pm|<\frac C{|\lambda|},\,\textit{ as $|\lambda|\to\infty$.}\]and $C$ depends continuously on $x$, $y$.
\end{lemma}  
\begin{proof} 
By formula (\ref{E:eigen+}),  we have 
\begin{equation}
\Psi(x,y,\lambda)=1+\frac 1{2\pi }\int_{-\infty}^\infty e^{i\xi x}\widehat W(\xi,y,\lambda)d\xi.\label{E:csdbd2}
\end{equation} Write 
\[\widehat W(\xi,y,\lambda)=\frac 1\lambda A(\xi,y,\lambda).
\] Note $\widehat W\in \widehat{\mathbb X}$ with $\widehat{\mathbb X}$ defined by Definition \ref{D:SMfunction}. Therefore Theorem \ref{T:lambda2} implies 
\begin{equation}\textit{$A$ is uniformly bounded in $\widehat{\mathbb X}$.}\label{E:csdbd}
\end{equation}
Now we define
\begin{equation}
\begin{array}{ll}\frac {B_1(\xi,y,\lambda)}\lambda=\int_{-\infty}^y e^{i\lambda\xi(y-y')}\widehat {\partial_xQ}(\xi,y')dy',&\textsl{ if $\lambda\in\mathbb{C}^+$, $\xi\ge 0$;}\\
\frac {B_2(\xi,y,\lambda)}\lambda=-\int_{y}^\infty e^{i\lambda\xi(y-y')}\widehat {\partial_xQ}(\xi,y')dy',&\textsl{ if $\lambda\in\mathbb{C}^+$, $\xi\le 0$;}\\
\frac {B_3(\xi,y,\lambda)}\lambda=-\int_{y}^\infty e^{i\lambda\xi(y-y')}\widehat {\partial_xQ}(\xi,y')dy',&\textsl{ if $\lambda\in\mathbb{C}^-$, $\xi\ge 0$;}\\
\frac {B_4(\xi,y,\lambda)}\lambda=\int_{-\infty}^y e^{i\lambda\xi(y-y')}\widehat {\partial_xQ}(\xi,y')dy',&\textsl{ if $\lambda\in\mathbb{C}^-$, $\xi\le 0$.}\end{array}
\label{E:csdWK}
\end{equation}
By (\ref{E:WK}),  (\ref{E:K}), (\ref{E:csdbd}), (\ref{E:csdWK}), and Theorem \ref{T:lambda2}, we obtain
\begin{equation}
\textit{$B_1$, $B_2$, $B_3$, $B_4$ are uniformly bounded in $\widehat{\mathbb X}$}.\label{E:csdbd1}
\end{equation}

Differentiating both sides of (\ref{E:WK}), we obtain
\begin{eqnarray}
&&(1-\mathcal K_\lambda)\partial_\lambda\widehat W\label{E:csdbd4}\\
=&&iy\left[\int^{y}_{-\infty} e^{i\lambda\xi(y-y')}\xi\left(\widehat {\partial_xQ}\ast\widehat W\right)dy'
+\int^{y}_{-\infty} e^{i\lambda\xi(y-y')}\xi \widehat {\partial_xQ}dy'\right]\nonumber\\
-&&i\left[\int^{y}_{-\infty} e^{i\lambda\xi(y-y')}y'\xi\left(\widehat {\partial_xQ}\ast\widehat W\right)dy'
+\int^{y}_{-\infty} e^{i\lambda\xi(y-y')}y'\xi \widehat {\partial_xQ}dy'\right]\nonumber\\
=&&iy\left[\int^{y}_{-\infty} e^{i\lambda\xi(y-y')}\left(\widehat {\partial_x^2Q}\ast\widehat W\right)dy'
+\int^{y}_{-\infty} e^{i\lambda\xi(y-y')} \widehat {\partial_x^2Q}dy'\right]\nonumber\\
-&&i\left[\int^{y}_{-\infty} e^{i\lambda\xi(y-y')}y'\left(\widehat {\partial_x^2Q}\ast\widehat W\right)dy'
+\int^{y}_{-\infty} e^{i\lambda\xi(y-y')}y' \widehat {\partial_x^2Q}dy'\right]\nonumber\\
+&&iy\int^{y}_{-\infty} e^{i\lambda\xi(y-y')}\left(\widehat {\partial_xQ}\ast\xi\widehat W\right)dy'
-i\int^{y}_{-\infty} e^{i\lambda\xi(y-y')}y'\left(\widehat {\partial_xQ}\ast\xi\widehat W\right)dy'\nonumber
\end{eqnarray}for $\lambda\in \mathbb C^+$, $\xi\ge 0$ (Other cases can be done similarly). Define
\begin{eqnarray*}
&&\frac {C_1(\xi,y,\lambda)}\lambda=\int^{y}_{-\infty} e^{i\lambda\xi(y-y')}\widehat {\partial_x^2Q}(\xi,y')dy',\\
&&\frac {C_2(\xi,y,\lambda)}\lambda=\int^{y}_{-\infty} e^{i\lambda\xi(y-y')}\widehat {y'\partial_x^2Q}(\xi,y')dy',\\
&&\frac {C_3(\xi,y,\lambda)}\lambda=\xi\widehat W(\xi,y,\lambda).
\end{eqnarray*}
Using the definition of ${\mathbb P}_{\infty,k,1}$,  and following the way to prove (\ref{E:csdbd1}), then
one can show that
\begin{equation}
\textit{$C_1$, $C_2$, $C_3$ are uniformly bounded in $\widehat{\mathbb X}$ }
\label{E:csdWK1}
\end{equation}
if $Q\in {\mathbb P}_{\infty,k,1}$ and $k\ge 6$. Combining (\ref{E:csdbd2}), (\ref{E:csdbd}), (\ref{E:csdbd4}), (\ref{E:csdWK1}),  and (\ref{E:K}),  we prove $|\partial_\lambda\Psi_\pm|<\frac C{|\lambda|}$ as $|\lambda|\to\infty$ and $C$ depends continuously on $x$, $y$.

Since $\partial_x\Psi(x,y,\lambda)=\frac i{2\pi }\int_{-\infty}^\infty e^{i\xi x}\xi\widehat W(\xi,y,\lambda)d\xi
$. Modifying the above argument and letting $k\ge 7$ in ${\mathbb P}_{\infty,k,1}$, one can obtain the estimate for $|\partial_\lambda\partial_x\Psi_\pm|$ as well. 
\end{proof}

\begin{lemma}\label{L:lambdal2}
If $Q\in {\mathbb P}_{\infty,k,1}$, $k\ge 7$, and $Z=\phi$, then
\[
|\partial_\lambda\Psi_\pm|<\frac C{|\lambda|},\,\textit{ as $|\lambda|\to\infty$,}\]and $C$ depends continuously on $x$, $y$.
\end{lemma}  
\begin{proof} Since the property we wish to justify is a local property. Without loss of generality, we need only to show
\begin{equation}
|\chi(x,y)\partial_\lambda\Psi_\pm|<\frac C{|\lambda|},\,\textit{ as $|\lambda|\to\infty$.}\label{E:reduce}
\end{equation}
Where $C$ depends continuously on $x$, $y$, and $\chi(x,y)$ is any fixed smooth function with compact support. Now by the induction scheme as the proof of Theorem \ref{T:LNexistence}, we have
\begin{equation}
\begin{array}{l}\Psi(x,y,\lambda)=\begin{cases}\Psi^-(x,y,\lambda)a^-(x,y,\lambda),& y\le 0,\\
\Psi^+(x,y,\lambda)a^+(x,y,\lambda),& y\ge 0,
\end{cases}
\end{array}\label{E:BP3}
\end{equation}
and
\[
\partial_\lambda\Psi=\left(\partial_\lambda\Psi^\pm\right)a^\pm+\Psi^\pm\partial_\lambda a^\pm.
\]
By induction and applying Lemma \ref{L:boundedness}, and \ref{L:lambdal2'}, it reduces to showing
\[|\chi(x,y)\partial_\lambda a
|<\frac C{|\lambda|} \,\textit{ as }|\lambda|\to\infty.\]Where 
\[
\begin{array}{l}a(x,y,\lambda)=\begin{cases}a^-(x,y,\lambda),& y\le 0,\\
a^+(x,y,\lambda),& y\ge 0,
\end{cases}
\end{array}
\]

By (\ref{E:BP3}), one can derive  the inhomogeneous Riemann-Hilbert problem
\[
\left(\chi\partial_\lambda a\right)_{+,z}(x,0,\lambda)=g (x,\lambda)+\left(\Psi^+\right)^{-1}\Psi^-\left(\chi\partial_\lambda a\right)_{-,z}(x,0,\lambda),\]
with
\[
g=\left[\partial_\lambda\left((\Psi^+)^{-1}\Psi^-\right)\right]\chi a_{-,z}.\]
Hence  \cite{AF97}
\begin{equation}
\chi\partial_\lambda a(x,y,\lambda)=\tilde\Psi(x,y,\lambda)^{-1}\left[\frac 1{2\pi i}\int_{\mathbb R}\frac {\left(\Psi^+\right)(t,0,\lambda)g(t,\lambda)}{t-z}dt\right]\label{E:chi}
\end{equation}
with $x+\lambda y=z$, and
\[
\begin{array}{l}\tilde\Psi(x,y,\lambda)=\begin{cases}\Psi^-(x,y,\lambda),& y\le 0,\\
\Psi^+(x,y,\lambda),& y\ge 0.
\end{cases}
\end{array}\] 
Therefore by Lemma \ref{L:lambdal2'} and (\ref{E:chi}),
\begin{eqnarray}
|\chi\partial_\lambda a|_{L_2(\mathbb R, dx)}<&&C|\tilde\Psi^{-1}||\int_{\mathbb R}\frac {\left(\Psi^+\right)(t,0,\lambda)g(t,\lambda)}{t-z}dt|\label{E:for1}\\
<&&C|\tilde\Psi^{-1}||\int_{\mathbb R}\frac {\chi(t,0)\left(\Psi^+\right)(t,0,\lambda)g(t,\lambda)}{t-z}dt|\nonumber\\
<&&C|g||\int_{\mathbb R}\frac {\chi(t,0)}{t-z}dt|_{L_2(\mathbb R, dx)}\nonumber\\
<&&\frac C{|\lambda|}\nonumber
\end{eqnarray} 
as $|\lambda|\to\infty$. Furthermore, differentiating both sides of (\ref{E:chi}) and using Corollary \ref{C:boundedness}, Lemma \ref{L:lambdal2'}, we obtain
\begin{eqnarray}
&&|\partial_x\left(\chi\partial_\lambda a\right)|_{L_2(\mathbb R, dx)}\label{E:for2}\\
<&&|\partial_x \tilde\Psi^{-1}||\left[\frac 1{2\pi i}\int_{\mathbb R}\frac {\left(\Psi^+\right)(t,0,\lambda)g(t,\lambda)}{t-z}dt\right]|_{L_2(\mathbb R, dx)}\nonumber\\
&&+|\tilde\Psi^{-1}||\partial_x\left[\frac 1{2\pi i}\int_{\mathbb R}\frac {\left(\Psi^+\right)(t,0,\lambda)g(t,\lambda)}{t-z}dt\right]|_{L_2(\mathbb R, dx)}\nonumber\\
<&&\frac C{|\lambda|}\nonumber
\end{eqnarray}
Hence the lemma follows from (\ref{E:for1}), (\ref{E:for2}), and Sobolev's theorem.
\end{proof}

We conclude this section by \textbf{the proof of Theorem \ref{T:CSDsum}} and the definition of continuous scattering transformation.

\begin{proof} The condition (\ref{E:0ana15'}) follows from Lemma \ref{L:qxx'}. 
The identity (\ref{E:real1'}) comes from (\ref{E:det2}) and Lemma \ref{L:qxx'}. Besides, (\ref{E:reality2}) and Lemma \ref{L:qxx'} imply that for $\lambda\in \mathbb{R}$
\[v(x+\lambda y,\lambda)=\Psi_-(x+\lambda y,\lambda)^{-1}\Psi_+(x+\lambda y,\lambda)=\Psi_+(x+\lambda y,\lambda)^*\Psi_+(x+\lambda y,\lambda).\]
Therefore (\ref{E:real2'}) follows. 

Next note  that Lemma \ref{T:CSDana} implies that
\begin{equation}
\partial_x^i\partial_y^j\left(\Psi_\pm-1\right)\,
\textit{ are uniformly bounded in }   L_\infty\cap L_2(\mathbb{R},d\lambda)\cap L_1(\mathbb{R},d\lambda).\label{E:0ana1'}
\end{equation}  So (\ref{E:0anal2'}) follows. Combining  Lemma \ref{L:boundedness}, (\ref{E:0ana1'}), one obtains $\partial_x^i\partial_y^j\left(v-1\right)\to 0$ uniformly in  $L_\infty$. So condition (\ref{E:0ana14'}) follows from (\ref{E:0anal2'}), and the Lebesque convergence theorem. Finally, condition (\ref{E:0anal3'}) is derived by applying Lemma \ref{L:lambdal2}.
\end{proof}

\begin{definition}\label{D:SMcd}
For $Q\in {\mathbb P}_{\infty,k,1}$, $k\ge 7$, if the eigenfunction $\Psi(x,y,\cdot)$ has limits $\Psi_\pm$ on $\mathbb{R}$, then we define the continuous scattering data of $Q$ to be $v\in \mathfrak{S}_{c,k}$ obtained by Theorem \ref{T:CSDsum}. Moreover, the continuous scattering transformation $\mathcal S_c$ on $Q$ is defined by $\mathcal S_c(Q)=v$.
\end{definition}

%%%%%%%%%%%%%%%%%%%%%%%%%%%%%%%%%%%%%%%%%%%%%%%%%%%%%%%%%%%%%%%%%%%%%%%%%%%%%%%%%%%%%%%%%%%%%%%%%%%%%%%%%%%%%%%%%%%%%%%%%%%%%%
\section{Inverse problem: Continuous scattering data} \label{S:CSDinv}
%%%%%%%%%%%%%%%%%%%%%%%%%%%%%%%%%%%%%%%%%%%%%%%%%%%%%%%%%%%%%%%%%%%%%%%%%%%%%%%%%%%%%%%%%%%%%%%%%%%%%%%%%%%%%%%%%%%%%%%%%%%%%%%

We first \textbf{prove Theorem \ref{T:invexistence}} by solving the Riemann-Hilbert problem via a modified scheme of Section \ref{S:ENS}. 
\begin{proof} First of all, (\ref{E:0anal2'}), (\ref{E:0ana14'}) and Lemma \ref{L:SMrm} imply  
that there exists a constant $M>0$ such that, as $|x|$ or $|y|>M-1$,  the Riemann-Hilbert problem $(\lambda\in\mathbb{R}, v(x, y, \lambda))$ can be solved and 
\begin{equation}
|\partial_x^i\partial_y^j\left(\Psi_\pm-1\right)|_{L_2(d\lambda)}\le C|v-1|_{L_2(d\lambda)}
\label{E:invana1'''}
\end{equation}
for a constant $C$. Hence (\ref{E:invbdry''}) holds as $|x|$ or $|y|>M-1$. 
Applying Holder inequality, (\ref{E:0anal2'}), (\ref{E:0ana14'}), and (\ref{E:invana1'''}), we then derive: 
\begin{gather*}
\textit{For each fixed $\lambda\notin\mathbb R$, $\forall |x|$ or $|y|>M-1$, }\\
\partial_x^i\partial_y^j\left(\Psi-1\right)  \in L_\infty(dxdy),\\
\partial_x^i\partial_y^j\left(\Psi-1\right)\to 0  \textit{ in $L_\infty(dxdy)$, as $x$ or $y\to \infty$}.
\end{gather*}
%Besides,  follows from Lemma \ref{L:SMrm}, (\ref{E:invana5'}), and the Sobolev's theorem 

Hence, to prove Theorem \ref{T:invexistence}, it is sufficient to solve the Riemann-Hilbert problem $(\lambda\in\mathbb{R}, v(x, y, \lambda))$ and establish (\ref{E:invbdry''}), (\ref{E:invbdry'})  for $\max(|x|, |y|)< M$. The scheme in Section \ref{S:ENS}, in particular Lemma \ref{L:trifact}-\ref{L:linearsystem}, can be adapted to the solving of this problem. More precisely, 

\begin{lemma}\label{L:invtrifact}
For $\lambda,\,x,\,y\in \mathbb{R}$, we have a factorization
\[
v( x, y,\lambda)=\left(1+h_l\right)^{-1}\chi\left(1+h_u\right) (x,y, \lambda), \]
and for $i+j\le k-4$,
\begin{eqnarray}
	&& \textsl{$\chi$ is diagonal and $h_u$ ($h_l$) is strictly upper (lower) triangular.}\label{E:invdec1}\\
	%&&\textsl{$\chi-1$, $h_u$, $h_l$, $\partial_{ x}\chi$,  $\partial_{ x}h_u$, $\partial_{ x} h_l$, $\partial_{ y}\chi$,  $\partial_{ y}h_u$,    $\partial_{ y} h_l$ are uniformly}\label{E:invdec2}\\
	%&& \textsl{bounded in $  L_\infty\cap L_2(\mathbb R,d\lambda)$.}\nonumber\\
		 && \textsl{$\partial_{ x}^i\partial_{ y}^j\left(\chi-1\right)$, $\partial_{ x}^i\partial_{ y}^jh_u$, $\partial_{ x}^i\partial_{ y}^jh_l$, $\partial_\lambda\chi$, $\partial_\lambda h_u$, $\partial_\lambda h_l$  are in $L_\infty\cap  L_2(\mathbb R, d\lambda)$ }\label{E:invdec4}\\
		%&&\textsl{  }\nonumber\\
		 &&\textit{and the norms depend continuously on $x$, $y$.}\nonumber\\
		&& \textsl{$\chi-1$, $h_u$, $h_l\to 0$ uniformly in  $ L_\infty\cap L_2(\mathbb R,d\lambda)$ as $|x|$ or $|y|\to\infty$.
	 }\label{E:invdec3}%\\
\end{eqnarray}
\end{lemma}
\begin{proof} We use the positivity condition (\ref{E:real2'}) to prove that $d^+_i$, $1\le i\le n$ vanishes nowhere for $\lambda\in\mathbb R$. Hence the statements can be proved by the same method as that in the proof of Lemma \ref{L:trifact}.
\end{proof}

\begin{lemma}\label{L:invdRH}
{\emph{(A diagonal Riemann-Hilbert problem)}}
For $\max(|x|, |y|)< M$,   there  exists uniquely a solution $\Xi(x,y,\lambda)$ to the Riemann-Hilbert problem $(\lambda\in\mathbb{R}, \chi)$ such that
\begin{eqnarray}
&&\textit{$\Xi-1$  are uniformly bounded in $ H^1(\mathbb{R},d\lambda)$; }\label{E:invdrh0}\\
&&\Xi-1,\, \partial_x\Xi,\,\partial_y\Xi \textit{ are uniformly bounded in $L_2(\mathbb R, d\lambda)$,}
\label{E:invdrh1}
	\end{eqnarray}
and for  each fixed $\lambda\notin\mathbb R$,
\begin{equation}
\partial_{ x}^i\partial_{ y}^j\Xi \in L_\infty(dxdy)\textit{  for $\max(|x|, |y|)< M$}.\label{E:invdrh5}%\\
\end{equation}
 \end{lemma}
\begin{proof}
Applying (\ref{E:invdec4}), and (\ref{E:invdec3}),  one obtains that 
\[
\textit{$\partial_\lambda^i\left(\chi-1\right)$ are uniformly bounded in $  L_\infty\cap L_2(\mathbb R,d\lambda)$, $i=0,1$.}
\]
Hence the winding number $N(x,y)=-\frac 1{2\pi i}\int \frac{d\arg \chi}{d\zeta}(x,y,\zeta)d\zeta$ is integer-valued. Moreover, the condition (\ref{E:invdec3}) implies that $N(x,y)\equiv 0$.

Thus for $\max(|x|, |y|)< M$, the existence of $\Xi$, and (\ref{E:invdrh0}) can be implied by  (\ref{E:invdec4}), %the compactness of $\max(|x|, |y|)\le M$, 
the  Sobolev's theorem,  and Lemma \ref{L:scalarrh}. 
By (\ref{E:invdec4}), (\ref{E:invdrh0}), and the formula 
\begin{eqnarray*}
&&\Xi(x,y,\lambda)=\exp\left\{\frac 1{2\pi i}\int_{\mathbb R}\frac{\log\chi(x,y,\zeta)}{\zeta-\lambda}d\zeta\right\},\\
&&\partial_{ x}\Xi(x,y,\lambda)=\exp\left\{\frac 1{2\pi i}\int_{\mathbb R}\frac{\log\chi(x,y,\zeta)}{\zeta-\lambda}d\zeta\right\}\left(\frac 1{2\pi i}\int_{\mathbb R}\frac{\partial_{ x}\chi(x,y,\zeta)}{\chi(x,y,\zeta)(\zeta-\lambda)}d\zeta\right),\\
&&\partial_{ y}\Xi(x,y,\lambda)=\exp\left\{\frac 1{2\pi i}\int_{\mathbb R}\frac{\log\chi(x,y,\zeta)}{\zeta-\lambda}d\zeta\right\}\left(\frac 1{2\pi i}\int_{\mathbb R}\frac{\partial_{ y}\chi(x,y,\zeta)}{\chi(x,y,\zeta)(\zeta-\lambda)}d\zeta\right),\\
&&\partial_{ x}^2\Xi(x,y,\lambda)=\exp\left\{\frac 1{2\pi i}\int_{\mathbb R}\frac{\log\chi(x,y,\zeta)}{\zeta-\lambda}d\zeta\right\}\left(\frac 1{2\pi i}\int_{\mathbb R}\frac{\partial_{ x}\chi(x,y,\zeta)}{\chi(x,y,\zeta)(\zeta-\lambda)}d\zeta\right)^2\\
&& \hskip.7in-\exp\left\{\frac 1{2\pi i}\int_{\mathbb R}\frac{\log\chi(x,y,\zeta)}{\zeta-\lambda}d\zeta\right\}\left(\frac 1{2\pi i}\int_{\mathbb R}\frac{\left(\partial_{ x}\chi(x,y,\zeta)\right)^2}{\chi^2(x,y,\zeta)(\zeta-\lambda)}d\zeta\right)\\
&&\hskip.7in+\exp\left\{\frac 1{2\pi i}\int_{\mathbb R}\frac{\log\chi(x,y,\zeta)}{\zeta-\lambda}d\zeta\right\}\left(\frac 1{2\pi i}\int_{\mathbb R}\frac{\partial_{ xx}\chi(x,y,\zeta)}{\chi(x,y,\zeta)(\zeta-\lambda)}d\zeta \right)
\\
&&\\
&&\cdots\cdots\cdots\cdots\cdots\cdots\cdots\cdots\cdots\cdots,
\end{eqnarray*}
we derive (\ref{E:invdrh1}). 
Finally, we obtain (\ref{E:invdrh5}) %and (\ref{E:invdrh3}) 
by  Holder inequality. 
\end{proof}

\begin{lemma}\label{L:invapprox}
For $\max(|x|, |y|)< M$, there exists a function $H(x,y,\lambda)$ satisfying
\[H=\begin{cases} H_{u}(x,y,\lambda), &\textit{for $\lambda\in \mathbb{C}^+$,}\\
H_{l}(x, y,\lambda),& \textit{for $\lambda\in\mathbb{C}^-$}, \end{cases}
\]
and
\begin{itemize}
\item $H(x,y,\lambda)\in L_\infty\cap H^1(\mathbb R, d\lambda)$, and $\partial_x^i\partial_y^j H(x,y,\lambda)\in L_\infty\cap L_2(\mathbb R, d\lambda)$
	\item $|\Xi_{-}(1+H_{-})v(1+H_{+})^{-1}\Xi_{+}^{-1}(x,y,\lambda)-1|_{H^1(\mathbb{R}, d\lambda)}<\infty$. 
	\item $|\Xi_{-}(1+H_{-})v(1+H_{+})^{-1}\Xi_{+}^{-1}(x,y,\lambda)-1|_{L_\infty}\Vert C_\pm\Vert<1$.
	\item $H_{u}$ ($H_{l}$) is strictly upper (lower) triangular.
	\item $H$ is rational in $\lambda\in\mathbb{C}^\pm$, with only simple poles and each corresponding residue is off diagonal,  with only one nonzero entry $\kappa$ and $\partial_{ x}^i\partial_{ y}^j\kappa\in L_\infty(dxdy)$. %Moreover, the non-zero entry tends to $0$ as $|x|, \,|y|\to M$ and $M\to\infty$.
\end{itemize}
\end{lemma}

\begin{proof} 
Combining (\ref{E:invdrh0}) with the results of Lemma \ref{L:invtrifact}, \ref{L:invdRH}, and the same method as in the proof of Lemma \ref{L:approx}, the lemma can be proved. %We only remark that by using (\ref{E:invana1'}), the non-zero entry tends to $0$ as $|x|, \,|y|\to\infty$.
\end{proof}  

\begin{lemma}\label{L:invRHsmdata}{\emph{ (A Riemann-Hilbert problem with small data)}} 
For $\max(|x|, |y|)< M$, the Riemann-Hilbert problem $(\lambda\in\mathbb{R}, \Xi_{-}(1+H_{-})v(1+H_{+})^{-1}\Xi_{+}^{-1})$ admits a solution $\varphi_{s}(x,y,\lambda)$. 
Moreover, 
\[\varphi_{s}-1,\, \partial_x\varphi_{s},\,\partial_y\varphi_{s} \textit{ are uniformly bounded in $L_2(\mathbb R, d\lambda)$,}\]
and for each fixed $\lambda\notin\mathbb R$,
\begin{gather*}
\partial_{ x}^i\partial_{ y}^j\left(\varphi_{s}-1\right)  \in L_\infty(dxdy).\label{E:invdrh5'}%\\
\end{gather*}
\end{lemma}
\begin{proof} The existence of the solution and its properties can be proved by Lemma \ref{L:SMrm}, \ref{L:invtrifact}, \ref{L:invdRH}, \ref{L:invapprox}, the property of the Cauchy operator $\mathcal C$ and Holder inequality.
\end{proof}

\begin{lemma}\label{L:invFORHP}{\emph{ (Factorization of the Riemann-Hilbert problem)}}   
 Suppose $\Psi(x,y,\lambda)$ satisfies Theorem \ref{T:invexistence}.
Then for $\max(|x|, |y|)< M$, there exists a unique function $u$,
\begin{equation}
u(x,y,\lambda)=1+\sum_{k=1}^{N} (\lambda-\lambda_k)^{-1}a_k(x,y),\label{E:invrational1}
\end{equation}
and
\begin{eqnarray}
&&\textit{$\partial_{ x}^i\partial_{ y}^ja_k\in L_\infty(dxdy)$,}
 %$a_k(x,y)\to 0$ as $|x|,\,|y|\to \infty$,}
 \label{E:invrational3}\\
&&\Psi(x,y,\lambda)=u\varphi_{s}\Xi (1+H).\label{invrational}
\end{eqnarray}

Conversely, if for $\max(|x|, |y|)< M$, $\exists u(x,y,\lambda)$ satisfying (\ref{E:invrational1}), (\ref{E:invrational3}) and
\begin{equation}
\textit{$u\varphi_{s}\Xi (1+H)$ is holomorphic for $ \lambda\in \mathbb{C}^\pm$,}\label{E:invrational4}
\end{equation}
Define $\Psi=u\varphi_{s}\Xi (1+H)$ for $\max(|x|, |y|)< M$. 
Hence $\Psi$ satisfies Theorem \ref{T:invexistence}.
\end{lemma}

We then use Lemma \ref{L:invFORHP} to prove Theorem \ref{T:invexistence}:

{\emph{ (a) A linear system for $u(x,y,\lambda)$:}}  

Let 
\begin{equation}
u(x,y,\lambda)=1+\sum_{k=1}^p (\lambda-\lambda_k)^{-1}a_k.\label{E:invform}
\end{equation}
Then at $\lambda_j$
\begin{equation}
u(x,y,\lambda)=(\lambda-\lambda_j)^{-1}a_j+b_j+O(|\lambda-\lambda_j|),\label{E:invform'}
\end{equation}
with
\begin{equation}
 b_j=1+\sum_{k\ne j}(\lambda_j-\lambda_k)^{-1}a_k. \label{invbj}
\end{equation}

Since $\lambda_j$ is a simple pole of $H$ and $\varphi_s\Xi$ is regular at $\lambda_j$. We can write
\begin{eqnarray}
&&1+H(x,y,\lambda)=(\lambda-\lambda_j)^{-1}h_j+n_j+O(|\lambda-\lambda_j|),\label{E:inv1+R}\\
&&\varphi_{s}\Xi (x,y,\lambda)=\alpha_j+\beta_j(\lambda-\lambda_j)+O(|\lambda-\lambda_j|^2).\label{E:invfd}
\end{eqnarray}

We then try to find $a_k$, such that $u(x,y,\lambda)\varphi_{s}(x,y,\lambda)\Xi(x,y, \lambda)(1+H(x,y,\lambda)) $  
is holomorphic at $\lambda_j$. This yields the linear system for $a_k$:
\begin{eqnarray}
&&a_j\alpha_jh_j=0,\qquad\qquad\qquad\qquad\,\,1\le j\le p,\label{invbcd1}\\
&&b_j\alpha_jh_j+a_j(\beta_jh_j+\alpha_jn_j)=0, \,1\le j\le p.\label{invbcd2}
\end{eqnarray}

{\emph{ (b) Solving the linear system (\ref{invbcd1})-(\ref{invbcd2}):}}  

Note by Lemma \ref{L:invapprox}, one can conclude 
\begin{equation}
(h_jn_j^{-1})^2=0. \label{E:n2}
\end{equation}
Therefore, 
it can be justified that (\ref{invbcd1}) is a consequence of (\ref{invbcd2}). Note the off-diagonal form of $h_l$ ($h_u$) in Lemma \ref{L:invtrifact} is crucial here.

Inserting (\ref{invbj}) into (\ref{invbcd2}), we obtain a system of $pn^2$ linear equations in $pn^2$ unknowns (the entries of $a_k$ with coefficients in entries of $h_j(x,y)$, $n_j(x,y)$, $\alpha_j(x,y)$, $\beta_j(x,y)$. Therefore, we conclude the existence problem of $\Psi$ is Fredholm.

{\emph{ (c) Solving the Riemann-Hilbert problem:}} 

Using the Fredholm alternality, we need only to show that: for any fixed $x,\,y$ the homogeneous problem (with limit $0$ rather than $1$ as $\lambda\to \infty$) has only the trivial solution. Suppose $f(x,y,\lambda)$ solves this homogeneous problem. Consider $g(x,y,\lambda)=f(x,y,\lambda)f(x,y,\bar\lambda)^*$. Since $f(x,y,\cdot)\in L_2(\mathbb R, d\lambda)$, we have $g(\lambda)\in L_1(\mathbb{R},d\lambda)$ and is holomorphic in $\mathbb C^\pm$. Thus the Cauchy's theorem implies
\[0=\int_\mathbb{R}g_+(s)ds=\int_\mathbb{R}f_+(s)f_-(s)^*ds=\int_\mathbb{R}f_-(s)v(s)f_-(s)^*ds.\]
Because of (\ref{E:real2'}) we conclude $f_-\equiv 0$ on $\mathbb{R}$, so also $f_+\equiv 0$ and $f\equiv 0$. 

Hence we prove the solvability of the Riemann-Hilber problem in  Theorem \ref{T:invexistence}. 
%Note the condition  (\ref{E:invbdry'}) follows from Lemma \ref{L:invtrifact}-\ref{L:invRHsmdata}, the rational form (\ref{E:invform}) and (\ref{invrational}). Hence (\ref{E:invbdry}) can be proved as well.
\end{proof}

\begin{lemma}\label{L:invsym}
For the solution $\Psi$ of the Riemann-Hilbert problem obtained in Theorem \ref{T:invexistence}, we have
\begin{eqnarray}
&&\det \Psi(x,y,\lambda)\equiv 1,\label{E:invsym0}\\
&&\Psi(x,y,t,\lambda)\Psi(x,y,t,\bar\lambda)^*\equiv 1.\label{E:invsym1}
\end{eqnarray}
\end{lemma}

\begin{proof} 
By (\ref{E:real1'}), $\det\Psi(x,y,\cdot)$ has no jump across the real line. So applying the Liouville's theorem, (\ref{E:invsym0}) follows from the the holomorphic property in $\mathbb C^\pm$ and $\Psi\to 1$ as $|\lambda|\to\infty$. Hence 
 $\Psi(x,y,\lambda)$ is invertible for all $\lambda\in\mathbb C$, limits $(\Psi(x,y,z,\bar\lambda)^*)^{-1}_\pm$ for $\lambda\in \mathbb R$ exist, and $(\Psi(x,y,z,\bar\lambda)^*)^{-1}$ fulfills the boundary condition as $|\lambda|\to \infty$.

Secondly, by (\ref{E:real2'}) and $\Psi_+=\Psi_-v$, we obtain
\begin{equation}
\begin{array}{lll}
\left(\Psi(x,y,\bar\lambda)^*\right)_+ &=\Psi_-(x,y,\bar\lambda)^*
&=\left(\Psi_+(x,y,\bar\lambda)v^{-1}\right)^*\\
{}&=v^{-1}\Psi_+(x,y,\bar\lambda)^*&=v^{-1}\left(\Psi(x,y,\bar\lambda)^*\right)_-.
\end{array}\label{E:invsym3'}
\end{equation}
So 
 \[\left(\left(\Psi(x,y,\bar\lambda)^*\right)_+\right)^{-1}=\left(\left(\Psi(x,y,\bar\lambda)^*\right)_-\right)^{-1}v.\]  
Therefore $\left(\Psi(x,y,\bar\lambda)^*\right)^{-1}$ satisfies the same Riemann-Hilbert problem in Theorem \ref{T:invexistence}. Consequently $\Psi(x,y,\lambda)=\left(\Psi(x,y,\bar\lambda)^*\right)^{-1}$  by the uniqueness property of Theorem \ref{T:invexistence} (the Liouville's theorem) and (\ref{E:invsym1}) is established.
\end{proof}

We conclude this section by \textbf{the proof of Theorem \ref{T:invcd}} and the definition of inverse scattering transformation.

\begin{proof}
By (\ref{E:invbdry}), the boundary condition  (\ref{y-bdry}) is satisfied. Besides, the Cauchy integral formula, and Theorem \ref{T:invexistence} imply
\begin{equation}
\Psi(x,y,\lambda)= I+\mathcal {C}\Psi_-(v-1).\label{E:invcauchy}
\end{equation}
For fixed $x,\,y\in\mathbb R$, applying $\mathcal L_\lambda=\partial_y-\lambda\partial_x$ to (\ref{E:invcauchy}) and using  (\ref{E:invbdry''}), (\ref{E:0anal2'}), we obtain
\begin{eqnarray}
\mathcal L_\lambda\Psi
 &=& \mathcal L_\lambda \mathcal{C}\Psi_-(v-1)\nonumber\\
&=& \mathcal{C}(\mathcal L_\zeta\Psi_-)(v-1)+\left[\mathcal L_\lambda,\mathcal{C}\right]\Psi_-(v-1)\nonumber\\
&=&\partial_x\left(\frac 1{2\pi i}\int_\mathbb{R}\Psi_-(x,y,\zeta)(v(x+\zeta y,\zeta)-1)d\zeta\right)\nonumber\\
&& +\mathcal{C}(\mathcal L_\zeta\Psi_-)(v-1).\nonumber\\
& =&\partial_xQ(x,y)+ \mathcal{C}(\left[\mathcal L_\zeta\Psi\right]_-)(v-1)\label{E:invcauchy'}
\end{eqnarray}
with 
$Q(x,y)$ given by (\ref{E:invq}). 
Hence comparing (\ref{E:invcauchy}) and (\ref{E:invcauchy'}) and using the uniqueness result of Theorem \ref{T:invexistence}, we obtain (\ref{E:Lax1}).

Besides,  (\ref{E:0anal2'}), (\ref{E:invbdry''}), %(\ref{E:invbdry'}),  
(\ref{E:invq}), and Holder inequality show that $Q$, $\partial_x Q$, and $\partial_y Q\in L_\infty $. Furthermore, by (\ref{E:Lax1}),  (\ref{E:invbdry'}), (\ref{E:invsym0}), and the $\lambda$-independence of $Q$, we derive  $\partial_{ x}^i\partial_{ y}^j Q\in L_\infty$ and  $\partial_{ x}^i\partial_{ y}^j Q$, $\partial_{ y} Q$, $Q\to 0$ as $x$ or $y\to\infty$, for $i+j\le k-4$, $i>0$.

Finally, by (\ref{E:invsym1}) and (\ref{E:Lax1}), we have
\begin{eqnarray*}
&&\left(\partial_x Q\right){\Psi(x,y,t,\overline\lambda)^*}^{-1}\\
=&&(\partial_y-\lambda\partial_x){\Psi(x,y,t,\overline\lambda)^*}^{-1}\\
=&&-{\Psi(x,y,t,\overline\lambda)^*}^{-1}\left((\partial_y-\lambda\partial_x)\Psi(x,y,t,\overline\lambda)^*\right){\Psi(x,y,t,\overline\lambda)^*}^{-1}\\
=&&-{\Psi(x,y,t,\overline\lambda)^*}^{-1}\left(\overline{(\partial_y-\overline\lambda\partial_x)\Psi(x,y,t,\overline\lambda)}^T\right){\Psi(x,y,t,\overline\lambda)^*}^{-1}\\
=&&-{\Psi(x,y,t,\overline\lambda)^*}^{-1}(\left(\partial_x Q\right)\Psi(x,y,t,\overline\lambda))^*{\Psi(x,y,t,\overline\lambda)^*}^{-1}\\
=&&-\left(\partial_x Q\right)^*{\Psi(x,y,t,\overline\lambda)^*}^{-1}.
\end{eqnarray*}Thus $\partial_x Q(x,y)\in su(n)$.
\end{proof}

\begin{definition}\label{D:IST}
For a function $v\in \mathfrak{S}_c$, we define the inverse scattering transformation $\mathcal S_c^{-1}$ on $v$ by $S_c^{-1}(v)=Q$, where $Q$ is obtained by Theorem \ref{T:invexistence}, and \ref{T:invcd}.
\end{definition}

%%%%%%%%%%%%%%%%%%%%%%%%%%%%%%%%%%%%%%%%%%%%%%%%%%%%%%%%%%%%%%%%%%%%%%%%%%%%%%%%%%%%%%%%%%%%%%%%%%%%%%%%%%%%%%%%%%%%%%%%%%%%%%
\section{The Cauchy problem: Continuous scattering data} \label{S:cauchy}
%%%%%%%%%%%%%%%%%%%%%%%%%%%%%%%%%%%%%%%%%%%%%%%%%%%%%%%%%%%%%%%%%%%%%%%%%%%%%%%%%%%%%%%%%%%%%%%%%%%%%%%%%%%%%%%%%%%%%%%%%%%%%%%
\textbf{We prove Theorem \ref{T:cauchy} in this section.}
\begin{proof} 
We can apply Theorem \ref{T:LNexistence} to find the eigenfunction $\Psi(x,y,0,\lambda)$. By assumption, and Theorem \ref
{T:CSDsum},  $\mathcal S_c(Q_0)\in \mathfrak{S}_{c,k}$.  

Now let us define $\textbf{v(t)}$ by
\begin{equation}\begin{array}{l}
\textbf{v}(t)=\{v(x,y,t,\lambda)=v(x+\lambda y+\lambda^2 t, \lambda)\}.
\end{array}\label{E:cauchySD}
\end{equation}
For each $t\in\mathbb R$, rewriting $x+\lambda y+\lambda^2 t=x+\lambda (y+\lambda t)=x+\lambda^2(t+\frac 1\lambda y)$ and modifying the approach in proving lemmas in Section \ref{S:DPSMasymp}-\ref{S:CSD}, one can justify that $\textbf{v}(t)\in \mathfrak{S}_{c,k}$ (see Definition \ref
{D:pcd}). So 
$v$ satisfies the algebraic constraints: 
\begin{itemize}
	\item $\det \,(v)\equiv 1$,
	\item $v= v^*>0$,
	\end{itemize}
and the analytic  constraints: for $i+j+h\le k-4$, 
\begin{itemize}
	\item $\mathcal L_\lambda v=0,\,\mathcal M_\lambda v=0$;
	\item $\partial_x^i\partial_y^j\partial_t^h\left(v-1\right)$ are uniformly bounded in $ L_\infty\cap L_2(\mathbb{R},d\lambda)\cap L_1(\mathbb{R},d\lambda)$;
	\item $\partial_x^i\partial_y^j\partial_t^h\left(v-1\right)\to 0$  uniformly in  $L_\infty\cap L_2(\mathbb{R},d\lambda)\cap L_1(\mathbb{R},d\lambda)$   as $|x|$ or $|y|$ or $t\to\infty$;
	\item $\partial_\lambda v\in  L_2(\mathbb{R},d\lambda)$ and the norms depend continuously on $x$, $y$.
\end{itemize}
Where 
$
\mathcal L_\lambda=\partial_y-\lambda\partial_x$, and $
\mathcal M_\lambda=\partial_t-\lambda\partial_y$.

Now we apply Theorem \ref{T:invexistence}, and \ref{T:invcd} to show the existence of $\Psi(x,y,t,\lambda)$ and $Q(x,y,t)$ satisfying (\ref{E:Lax1}), and (\ref{y-bdry}). More precisely, 
\begin{gather}
{\begin{array}{rl}\Psi(x,y,t,\lambda)=&I+\mathcal {C}\Psi_-(v-1)\\
=&I+\frac 1{2\pi i}\int_{\mathbb{R}}\frac{\Psi_-(x,y,\zeta)\left(v(x+\zeta y+\zeta^2 t, \zeta)-1\right)}{\lambda-\zeta}d\zeta
\end{array}}\label{E:chcauchy}\\
\Psi_\pm-1,\,\partial_x\Psi_\pm, \,\partial_y\Psi_\pm, \,\partial_t\Psi_\pm \textit{ are uniformly bounded in } L_2(d\lambda),\nonumber
\end{gather}
and for each fixed $\lambda\notin \mathbb R$, $i+j+h\le k-4$
\begin{equation}
\textit{$\partial_x^i\partial_y^j\partial_t^h\Psi \in L_\infty(dxdydt)$.}\label{E:psireg}
\end{equation}
In addition,
\begin{eqnarray*}
Q(x,y,t)=&& \frac 1{2\pi i}\int_\mathbb{R}\Psi_-(x,y,t,\zeta)(v(x+\zeta y+\zeta^2 t, \zeta)-1)d\zeta,
\end{eqnarray*} 
and for $i+j+h\le k-4$, $i^2+j^2>0$,
\begin{gather}
\textit{$\partial_x^i\partial_y^j \partial_t^hQ$, $\partial_tQ$, $Q\in L_\infty$,}\label{E:qreg1}\\
\textit{$\partial_x^i\partial_y^j\partial_t^hQ$, $\partial_tQ$, $Q \to 0$ in $ L_\infty$.}\label{E:qreg2}
\end{gather}

%Note for $\lambda\notin \mathbb R\cup Z$, by (\ref{E:chana2'})-(\ref{E:chana1''}), (\ref{E:chv'}), we have
%\begin{eqnarray*}
%&&\Psi-1,\,\, \Psi_x,\,\,\Psi_y ,\,\Psi_t \textit{ are uniformly bounded in $L_\infty(dxdy)$},\\
%&&Q,\,Q_x,\, Q_y\in L_\infty, \textit{ and $Q_x(x,y),\,Q_y(x,y)\in su(n)$.}
%\end{eqnarray*}

To prove (\ref{E:Lax02}), we note it is equivalent to prove
\begin{equation}
\mathcal M_\lambda\Psi=\left(\partial_yQ\right)(x,y,t)\Psi(x,y,t,\lambda).\label{E:equation}
\end{equation}
Applying $\mathcal M_\lambda$ to both sides of (\ref{E:chcauchy}) and using similar approach as that in the proof of Theorem \ref{T:invcd}, we obtain
\begin{equation}
\mathcal M_\lambda\Psi =\left(\partial_yQ\right)(x,y,t)+ \mathcal{C}(M_\zeta\Psi)_-(v-1)\label{E:chcauchy'}.
\end{equation}
Comparing (\ref{E:chcauchy}) and (\ref{E:chcauchy'}) and using the uniqueness result of Theorem \ref{T:invexistence}, we obtain (\ref{E:equation}). The smooth and decay properties of $Q$ can be derived by an argument similar to  the proof of Theorem \ref{T:invcd} and conditions (\ref{E:psireg})-(\ref{E:qreg2}).

Since we have obtain the differentiability of $\Psi(x,y,t,\lambda)$ and $Q(x,y,t)$. The compatibility condition of (\ref{E:Lax1}) and (\ref{E:equation}) yield (\ref{E:chiral}). 
\end{proof}

We conclude this report by a brief remark on examples of  $Q_0\in \mathbb P_{\infty,k,1}$, $k\ge 7$, and  the corresponding eigenfunction $\Psi_0$ has no poles. The first class of examples is $\mathbb P_1\cap \mathcal{S}$ ($\mathcal{S}$ is the set of Schwartz functions and $\mathbb P_1$ is defined by Definition \ref{D:SMfunction}). To construct an example with large norm, we let
$v(x,y,\lambda)=v(x+\lambda y,\lambda)$ satisfy
\[\det (v)=1,\quad
v=v^*>0,\quad v-1\in \mathcal{S},
\]
and for $\forall i$, $j,\,h\ge 0$, 
\begin{eqnarray*}
 &&\partial_x^i\partial_y^j\partial_\lambda^h\left(v-1\right)\in L_2({\rm R},\,d\lambda)\cap L_1({\rm R},\,d\lambda)\textit{ uniformly}, \\
 &&\partial_x^i\partial_y^j\partial_\lambda^h\left(v-1\right)\to 0 \,\,\textit{ in $L_2({\rm R},\,d\lambda)$ uniformly, as}\,\,|x|,\,|y|\to\infty.
\end{eqnarray*}
We can solve the inverse problem and obtain $\Psi_0\in \mathcal{S}$ by the argument in proving Theorem \ref{T:LNexistence}. Note here we need to use the reality condition $v=v^*>0$ to show the global solvability. 
Moreover, by using the fomula $Q_0(x,y)=\frac 1{2\pi i}\int_{\rm R}\psi_{0,-}(v-1)d\xi$, one obtains that $Q_0$ is Schwartz and possesses purely continuous scattering data.

\end{document}